\documentclass[11pt]{amsart}

\usepackage{fullpage}
\usepackage{amssymb}
\usepackage{amsmath}
\usepackage{amsxtra}
\usepackage{amscd}
\usepackage{graphicx}
\usepackage{amsfonts}
\usepackage{pb-diagram}
\usepackage{float}
\usepackage{enumerate}
\usepackage{dsfont}

\numberwithin{equation}{section}

\newtheorem{thm}{Theorem}
\newtheorem{lem}{Lemma}
\newtheorem{cor}{Corollary}
\newtheorem{prop}{Proposition}

\newtheorem{defn}{Definition}

\newtheorem{rem}{Remark}

\def\a{\mathsf{a}}
\def\b{\mathsf{b}}

\def\l{\mathsf{c}}
\def\m{\mathsf{d}}

\def\g{\gamma}
\def\r{\mathsf{r}}

\def\leq{\leqslant}

\def\Z{{\mathbb Z}}

\def\C{{\mathbb C}}
\def\N{\mathbb{N}}

\def\ot{\mathsf{t}}
\def\og{\mathsf{g}}
\def\oe{\mathsf{e}}

\def\spec{{\mathrm{Spec}}_{\rm JM}}
\def\cont{{\mathrm{Cont}}_d(n)}

\newcommand{\cc}{\mathrm{c}}
\newcommand{\ccc}{\mathrm{cc}}
\newcommand{\pos}{\mathrm{p}}
\def\cT{\mathcal{T}}
\def\bv{\mathbf{v}}
\newcommand{\blambda}{\boldsymbol{\lambda}}

\newcommand{\bmu}{\boldsymbol{\mu}}
\newcommand{\btau}{\boldsymbol{\tau}}

\newcommand{\btheta}{\boldsymbol{\theta}}
\newcommand{\bpsi}{\boldsymbol{\psi}}

\newcommand{\gad}{\Gamma_d}
\begin{document}

\title[Representation theory of the Yokonuma--Hecke algebra]
  {Representation theory of the Yokonuma--Hecke algebra}
\author{Maria Chlouveraki}
\address{Laboratoire de Math\'ematiques UVSQ, B\^atiment Fermat, 45 avenue des \'Etats-Unis,  78035 Versailles cedex, France.}
\email{maria.chlouveraki@uvsq.fr}

\author{Lo\"ic Poulain d'Andecy}
\address{Laboratoire de Math\'ematiques UVSQ, B\^atiment Fermat, 45 avenue des \'Etats-Unis,  78035 Versailles cedex, France.}
\email{loic.poulain-d-andecy@uvsq.fr}


\thanks{The research project is implemented within the framework of the Action ``Supporting Postdoctoral Researchers'' of the Operational Program
``Education and Lifelong Learning'' (Action's Beneficiary: General Secretariat for Research and Technology), and is co-financed by the European Social Fund (ESF) and the Greek State. The first author would also like to thank Guillaume Pouchin for fruitful conversations and remarks.}


\begin{abstract} 
We develop an inductive approach to the representation theory  of the Yokonuma--Hecke algebra ${\rm Y}_{d,n}(q)$, based on the study of the spectrum of its Jucys--Murphy elements which are defined here. We give explicit formulas for the irreducible representations of ${\rm Y}_{d,n}(q)$ in terms of standard $d$-tableaux; we then use them to obtain a semisimplicity criterion. Finally, we prove the existence of a canonical symmetrising form on ${\rm Y}_{d,n}(q)$ and calculate the Schur elements with respect to that form. 
\end{abstract}

\maketitle

\section{Introduction}

Yokonuma--Hecke algebras were introduced by Yokonuma \cite{yo} in the context of Chevalley groups, as a generalisation of Iwahori--Hecke algebras. More precisely, the Iwahori--Hecke algebra associated to a finite Chevalley group $G$ is the centraliser algebra associated to the permutation representation of $G$ with respect to a Borel subgroup of $G$. The Yokonuma--Hecke algebra is the centraliser algebra associated to the permutation representation of $G$ with respect to a maximal unipotent subgroup of $G$. 
Thus, the Yokonuma--Hecke algebra can be also regarded as a particular case of a unipotent Hecke algebra. In this paper, we study the representation theory of the Yokonuma--Hecke algebra of type $A$, obtained in the case where $G$ is the general linear group over a finite field.

In recent years, the presentation of the Yokonuma--Hecke algebra has been transformed  by Juyumaya  \cite{ju, ju2, ju3} to the one used in this paper; the new presentation is given by generators and relations, depending on two positive integers, $d$ and $n$, and a parameter $q$.  For 
$q^2=p^m$ and $d=p^m-1$, where $p$ is a prime number and $m$ is a positive integer, 
the Yokonuma--Hecke algebra, denoted by ${\rm Y}_{d,n}(q)$, is the centraliser algebra associated to the permutation representation of 
${\rm GL}_n(\mathbb{F}_{p^m})$ 
with respect to a maximal unipotent subgroup.

Now, the Yokonuma--Hecke algebra ${\rm Y}_{d,n}(q)$ is a quotient of the group algebra of the framed braid group $\Z \wr B_n$ and of the  modular framed braid group
$ (\Z/d\Z) \wr B_n$, where $B_n$ is the classical braid group on $n$ strands (of type $A$). Hence,  ${\rm Y}_{d,n}(q)$ has a natural topological interpretation in the context of framed braids and framed knots. In view of this, Juyumaya defined a Markov trace on ${\rm Y}_{d,n}(q)$ \cite{ju3}, which was later used by Lambropoulou and himself to define an invariant for framed knots \cite{jula1, jula2}. What is more, it was subsequently proved that this invariant can be extended to classical and singular knots \cite{jula3, jula4}; it is the only known knot invariant with this property.

Moreover, the Yokonuma--Hecke algebra ${\rm Y}_{d,n}(q)$ can be regarded as a deformation of the group algebra of the complex reflection group $G(d,1,n) \cong (\Z/d \Z) \wr \mathfrak{S}_n$, where $\mathfrak{S}_n$ denotes the symmetric group on $n$ letters. For $d=1$, the algebra  ${\rm Y}_{1,n}(q)$ coincides with the Iwahori--Hecke algebra $\mathcal{H}_n(q)$ of type $A$.

The Yokonuma--Hecke algebra is quite different from the other famous deformation of the group algebra of $G(d,1,n)$, the Ariki--Koike algebra \cite{AK}. 
As mentioned above, the Yokonuma--Hecke algebra is a quotient of the group algebra of the modular framed braid group
$ (\Z/d\Z) \wr B_n$. The Ariki--Koike algebra is a quotient of the group algebra of the braid group of type $B$. Thus, 
in some sense, 
the Yokonuma--Hecke algebra is a deformation where the wreath product structure of  $G(d,1,n)$ is respected; in 
the case of the Ariki--Koike algebra this structure is lost in view of the preservation of the classical Hecke quadratic relation. On the other hand, this preservation of the quadratic relation makes the Iwahori--Hecke algebra of type $A$ an obvious subalgebra of the Ariki--Koike algebra.
As far as the Yokonuma--Hecke algebra is concerned, the Iwahori--Hecke algebra $\mathcal{H}_n(q)$ of type $A$ is an obvious quotient of  ${\rm Y}_{d,n}(q)$,  but not an obvious subalgebra.

In this paper, we will study the representation theory of this object of high algebraic and topological interest. Some information on its representation theory in the general context of unipotent Hecke algebras has been obtained by Thiem in \cite{thi, thi2,thi3}. Unfortunately, the generality of his results and the change of  presentation for ${\rm Y}_{d,n}(q)$ do not allow us to put this information to practical use. 
In this paper, we develop an inductive, and highly combinatorial, approach to the representation theory of the Yokonuma--Hecke algebra. 
We give explicit formulas for all its irreducible presentations (in the semisimple case), which one can work with.
Thanks to our formulas we are able to obtain a semisimplicity criterion for ${\rm Y}_{d,n}(q)$, and calculate its Schur elements with respect to the canonical symmetrising form defined here, thus getting a glimpse of the modular representation theory of the Yokonuma--Hecke algebra. 

For $d=1$, the approach presented here coincides with the one in \cite{IO} for the Iwahori--Hecke algebra of type $A$. For $d\geq 1$, an inductive approach, in the spirit of \cite{OV}, was given in \cite{OPdA2} for the complex reflection group $G(d,1,n)$. Our approach for the Yokonuma--Hecke algebra can be seen as a deformed version of the approach in \cite{OPdA2} (different from the deformed version presented in \cite{OPdA1} for the Ariki--Koike algebra).

The paper is organised as follows: in Section $2$, we define the Yokonuma--Hecke algebra ${\rm Y}_{d,n}(q)$, and we suggest 
analogues for  ${\rm Y}_{d,n}(q)$ of the Jucys--Murphy elements of the symmetric group; these Jucys--Murphy elements for  ${\rm Y}_{d,n}(q)$ form a commutative family of elements.
We note that there are two types of Jucys--Murphy elements for  ${\rm Y}_{d,n}(q)$, similarly to the case of $G(d,1,n)$ and differently from the case of the Ariki--Koike algebra.

In Section $3$, we define  the affine Yokonuma--Hecke algebra $\widehat{\rm Y}_{d,n}(q)$, and we study a special class of representations in the simplest non-trivial case, when $n=2$. We use these representations to obtain information on the spectrum of the Jucys--Murphy elements of $\widehat{\rm Y}_{d,n}(q)$, in Section $4$.  This information allows us to show that each element in the spectrum of the Jucys--Murphy elements corresponds to a standard $d$-tableau of size $n$ (Propositions  \ref{spec-cont} and \ref{cont-tab}).

In Section $5$, we  give explicit formulas for the irreducible representations of  ${\rm Y}_{d,n}(q)$;  these formulas originate in the study of the simplest non-trivial affine Yokonuma--Hecke algebra $\widehat{\rm Y}_{d,2}(q)$. We show that the irreducible representations of  ${\rm Y}_{d,n}(q)$ are parametrised by the $d$-partitions of $n$. 
For a $d$-partition $\blambda$ of $n$, the irreducible representation $V_{\blambda}$ has a basis indexed by the standard $d$-tableaux of shape $\blambda$; thus, our formulas are in the combinatorial terms of standard $d$-tableaux.
We obtain that the  spectrum of the Jucys--Murphy elements  of  ${\rm Y}_{d,n}(q)$ is, in fact, in bijection with the set of standard $d$-tableaux of size $n$.
We also provide the branching rules for the Yokonuma--Hecke algebra, we construct a complete system of primitive idempotents and we show that the Jucys--Murphy elements generate a maximal commutative subalgebra of ${\rm Y}_{d,n}(q)$.
In Section $6$, we use the explicit formulas for the irreducible representations developed in Section $5$ to obtain a semisimplicity criterion for ${\rm Y}_{d,n}(q)$.

Finally, in Section $7$, we define a ``canonical'' symmetrising form $\btau$ on  ${\rm Y}_{d,n}(q)$, which becomes the canonical symmetrising form on the group algebra of $G(d,1,n)$ for $q=1$. 
We then determine the Schur elements for ${\rm Y}_{d,n}(q)$ with respect to $\btau$; these are powerful tools in the study of the modular representation theory of symmetric algebras.  In order to calculate them, we show that the Schur elements for ${\rm Y}_{d,n}(q)$ are products of Schur elements corresponding to Iwahori--Hecke algebras of type $A$, which are already known.

$ $\\
{\bf Piece of notation:} Throughout this paper, whenever $A$ is an algebra defined over a ring $R$ and $K$ is a field containing $R$, we will denote by $KA$ the tensor product $K \otimes_R A$.

\section{Jucys--Murphy elements of the Yokonuma--Hecke algebra}\label{sec-JM}

In this section we will give a presentation for the Yokonuma--Hecke algebra  ${\rm Y}_{d,n}(q)$ and we will define its Jucys--Murphy elements, which are the core of our construction of the representations of   ${\rm Y}_{d,n}(q)$.

\subsection{The Yokonuma--Hecke algebra ${\rm Y}_{d,n}(q)$}
Let  $d,\,n \in \N$ , $d \geq 1$. Let $q $ be an indeterminate.  The Yokonuma--Hecke algebra, denoted by ${\rm Y}_{d,n}(q)$, is a $\C[q,q^{-1}]$-associative algebra generated by the elements
$$
 g_1, \ldots, g_{n-1}, t_1, \ldots, t_n
$$
subject to the following relations:
\begin{equation}\label{modular}
\begin{array}{ccrclcl}
\mathrm{(b}_1)& & g_ig_j & = & g_jg_i && \mbox{for all  $i,j=1,\ldots,n-1$ such that $\vert i-j\vert > 1$,}\\
\mathrm{(b}_2)& & g_ig_{i+1}g_i & = & g_{i+1}g_ig_{i+1} && \mbox{for  all  $i=1,\ldots,n-2$,}\\
\mathrm{(f}_1)& & t_i t_j & =  &  t_j t_i &&  \mbox{for all  $i,j=1,\ldots,n$,}\\
\mathrm{(f}_2)& & t_j g_i & = & g_i t_{s_i(j)} && \mbox{for  all  $i=1,\ldots,n-1$ and $j=1,\ldots,n$,}\\
\mathrm{(f}_3)& & t_j^d   & =  &  1 && \mbox{for all  $j=1,\ldots,n$,}
\end{array}
\end{equation}
where $s_i$ is the transposition $(i, i+1)$, together with the quadratic
relations:
\begin{equation}\label{quadr}
g_i^2 = 1 + (q-q^{-1}) \, e_{i} \, g_i\qquad \mbox{for all $i=1,\ldots,n-1$}
\end{equation}
 where
\begin{equation}\label{ei}
e_i :=\frac{1}{d}\sum_{s=0}^{d-1}t_i^s t_{i+1}^{-s}.
\end{equation}

It is easily verified that the elements $e_i$ are idempotents in ${\rm Y}_{d,n}(q)$.  Also, that the elements $g_i$ are invertible, with
\begin{equation}\label{invrs}
g_i^{-1} = g_i - (q- q^{-1})\,e_i \qquad \mbox{for all $i=1,\dots,n-1$}.
\end{equation}

Moreover, note that we have 
\begin{equation}\label{ti}
t_{i+1}=g_i\,t_i\,g_i^{-1}=g_i^{-1}\,t_i\,g_i  \,\,\,\,\,\,\mbox{ for all $ i =1,\dots,n-1$.}
\end{equation}

If we specialise $q$ to $\pm1$, the defining relations (\ref{modular})--(\ref{quadr}) become the defining relations for the complex reflection group $G(d,1,n)$. Thus the algebra ${\rm Y}_{d,n}(q)$ is a deformation of the group algebra over $\C$ of the complex reflection group $G(d,1,n) \cong (\Z/d\Z) \wr \mathfrak{S}_n$, where $\mathfrak{S}_n$ is the symmetric group on $n$ letters.
Moreover, for $d=1$, the Yokonuma--Hecke algebra ${\rm Y}_{1,n}(q)$ coincides with the Iwahori--Hecke algebra $\mathcal{H}_n(q)$ of type $A$, and thus, for $d=1$ and $q$ specialised to $\pm1$, we obtain the group algebra over $\C$ of the symmetric group $\mathfrak{S}_n$.

The relations $({\rm b}_1)$, $({\rm b}_2)$, $({\rm f}_1)$ and $({\rm f}_2)$ are defining relations for the classical framed braid group $\mathcal{F}_n \cong \Z\wr B_n$, where $B_n$ is the classical braid group on $n$ strands, with the $t_j$'s being interpreted as the ``elementary framings" (framing 1 on the $j$th strand). The relations $t_j^d = 1$ mean that the framing of each braid strand is regarded modulo~$d$. Thus, the algebra ${\rm Y}_{d,n}(q)$ arises naturally  as a quotient of the framed braid group algebra over the modular relations ~$\mathrm{(f}_3)$ and the quadratic relations~(\ref{quadr}). Moreover, relations (\ref{modular}) are defining relations for the  modular framed braid group $\mathcal{F}_{d,n}\cong (\Z/d\Z) \wr B_n$, so the algebra ${\rm Y}_{d,n}(q)$ can be also seen  as a quotient of the modular framed braid group algebra over the  quadratic relations~(\ref{quadr}).

\begin{rem}{\rm 
Note that in all 
the papers \cite{ju3, jula2, jula3, chla} 
the parameter for the Yokonuma--Hecke algebra is denoted by $u$, and ${\rm Y}_{d,n}(u)$ is generated by elements $\tilde{g}_1, \ldots, \tilde{g}_{n-1}, t_1, \ldots, t_{n}$ satisfying relations (\ref{modular}) and the quadratic
relations:
\begin{equation}
\tilde{g}_i^2 = 1+ (u-1) \, e_{i} \, +(u-1) \, e_{i} \, \tilde{g}_i\qquad \mbox{for all $i=1,\ldots,n-1$}.
\end{equation}
By taking $u:=q^{2}$ and $g_i:= \tilde{g}_i+(q^{-1}-1)\, e_{i} \, \tilde{g}_i$ (and thus, $\tilde{g}_i = g_i+(q-1)\, e_{i} {g}_i$),
 we obtain our presentation of the Yokonuma--Hecke algebra. }
\end{rem}

\subsection{The split property of ${\rm Y}_{d,n}(q)$}
Due to the relations (\ref{modular})(${\rm f}_1$)--(\ref{modular})(${\rm f}_3$), every word $m$ in 
the generators $g_1,\ldots,g_{n-1},t_1,\ldots, t_n$ of 
${\rm Y}_{d,n}(q)$ can be written in the form
$$
m=t_1^{k_1}\ldots t_n^{k_n}\cdot \sigma
$$
where $k_1,\ldots,k_n \in \Z/d\Z$ and $\sigma$ is a word in $g_1,\ldots,g_{n-1}$. That is, $m$ splits into the `framing part'
$t_1^{k_1}\ldots t_n^{k_n}$ and the `braiding part' $\sigma$. 

Now let $w \in \mathfrak{S}_n$, and let $w=s_{i_1}s_{i_2}\ldots s_{i_r}$ be a reduced expression for $w$. Since the generators $g_i$ of the Yokonuma--Hecke algebra satisfy the same braid relations as the generators of $\mathfrak{S}_n$, Matsumoto's lemma implies that the element $g_w:=g_{i_1}g_{i_2}\ldots g_{i_r}$ is well-defined, that is, it does not depend on the choice of the reduced expression of $w$. 
Let $\ell$ denote the length function on $\mathfrak{S}_n$.
Then we have
\begin{equation}\label{rightmulti}
g_wg_{s_i}=\left\{\begin{array}{ll}
g_{ws_i}\,, & \text{if }\,\, \ell(ws_i) > \ell(w) \\
g_{ws_i}+(q-q^{-1})
g_we_i 
\,,  & \text{if }\,\, \ell(ws_i) < \ell(w)
\end{array}\right.\end{equation}
and
\begin{equation}\label{leftmulti}
g_{s_i}g_w=\left\{\begin{array}{ll}
g_{s_iw}\,, & \text{if }\,\, \ell(s_iw) > \ell(w) \\
g_{s_iw}+(q-q^{-1})e_ig_w\,,  & \text{if }\,\, \ell(s_iw) < \ell(w).
\end{array}\right.\end{equation}
Using the above multiplication formulas,
Juyumaya \cite{ju3} has proved that the following set is a  $\C[q,q^{-1}]$-basis for ${\rm Y}_{d,n}(q)$:
\begin{equation}\label{split}
\mathcal{B} =
\left\{\,t_1^{k_1}\ldots t_n^{k_n} g_w\,\left|\,\begin{array}{ll}w \in \mathfrak{S}_n, & k_1,\ldots,k_n \in \Z/d\Z
\end{array}\right\}\right.
\end{equation}
As a consequence, ${\rm Y}_{d,n}(q)$ is a free $\C[q,q^{-1}]$-module of rank $d^nn!$.

\subsection{Chain of algebras ${\rm Y}_{d,n}(q)$}
The algebras ${\rm Y}_{d,n}(q)$ form a chain, with respect to $n$, of algebras:
\begin{equation}\label{chain}
\C[q,q^{-1}]=:{\rm Y}_{d,0}(q)\subset {\rm Y}_{d,1}(q)\subset \ldots \subset {\rm Y}_{d,n-1}(q)\subset {\rm Y}_{d,n}(q)\subset\ldots\ \ ,
\end{equation}
where the injective morphisms are given by $${\rm Y}_{d,n-1}(q)\ni t_1,\dots,t_{n-1},g_1,\dots g_{n-2}\mapsto t_1,\dots,t_{n-1},g_1,\dots g_{n-2}\in {\rm Y}_{d,n}(q)\quad\ \text{for all $n\in \N$}.$$ 
The injectivity of the morphisms comes from the fact that there are $d^{n-1}(n-1)!$ elements in  $\mathcal{B}$ where $t_n$ and $g_{n-1}$ do not appear.

\subsection{Computation formulas in ${\rm Y}_{d,n}(q)$}\label{comp-for}
Let $i,k \in\{1,2,\ldots, n\}$ and set
\begin{equation}\label{edik}
e_{i,k} := \frac{1}{d} \sum_{s=0}^{d-1}t_i^st_{k}^{-s}.
\end{equation}
 Clearly
$e_{i,k} = e_{k,i}$ and it can be easily checked 
that $e_{i,k}^2 = e_{i,k}$. Note that $e_{i,i}=1$ and that $e_{i, i+1}=e_i$.
From $\mathrm{(f}_1)-\textrm{(f}_3)$, one obtains immediately that
\begin{equation}\label{edikrels}
\begin{array}{rcll}
t_i e_{j,k} & = & e_{j,k} t_i & \text{for $i,j,k=1,\dots,n$,}\\
e_{i,j} e_{k,l} & = & e_{k,l} e_{i,j} & \text{for $i,j,k,l=1,\dots,n$,}\\
e_{j,k}g_i  & = &  g_i e_{s_i(j),s_i(k)}& \text{for $i=1,\dots,n-1$ and $j,k=1,\dots,n$,}
\end{array}
\end{equation}
while one can also easily show that 
\begin{equation}\label{extraone}
\begin{array}{rcll}
t_ie_{i,k} & = & t_{k}e_{i,k}  &  \text{for $i,k=1,\dots,n$.}
\end{array}
\end{equation}
 Note that the last relation in (\ref{edikrels}) is also valid if $g_i$ is replaced by its inverse $g_i^{-1}$. 

\subsection{The Jucys--Murphy elements}
We define inductively the following elements in ${\rm Y}_{d,n}(q)$:
\begin{equation}\label{def-JM}
J_1:=1\ \ \ \text {and}\ \ \ J_{i+1}:=g_{i}\,J_{i} \,g_{i} \,\,\,\,\,\text{ for $i=1,\ldots,n-1$.}
\end{equation}
The element $J_{i+1}$ can be explicitly written in terms of the generators of ${\rm Y}_{d,n}(q)$ as:
\begin{equation}\label{form-JM}
J_{i+1}=1+(q-q^{-1})\Bigl(e_{i}\,g_{i}+e_{i-1,i+1}\,g_{i}g_{i-1}g_{i}+ \cdots \cdots+e_{1,i+1}\,g_{i}\ldots g_2g_1g_2\ldots g_{i} \Bigr)\ .
\end{equation}
Formula (\ref{form-JM}) is easily proved by induction on $i$, with the use of (\ref{def-JM}) and the last relation in (\ref{edikrels}).

We call the elements $J_1,\dots,J_n$, together with the elements $t_1,\dots,t_n$, the {\em Jucys--Murphy elements}
of the Yokonuma--Hecke algebra
 ${\rm Y}_{d,n}(q)$.

\begin{rem} {\rm For $q=\pm1$, the Yokonuma--Hecke algebra ${\rm Y}_{d,n}(q)$ specialises to the group algebra  over $\C$ of the complex reflection group $G(d,1,n)$. Formula (\ref{form-JM}) implies that the Jucys--Murphy elements for $G(d,1,n)$, defined in \cite{Ram-Sh,Wang} and used in \cite{OPdA2}, are the specialisations of the following elements of ${\rm Y}_{d,n}(q)$:
\[\frac{J_i-1}{q-q^{-1}}\ \ \ \ \ \text{for $i=1,\ldots,n$.}\]
Note that the above elements are well-defined elements of ${\rm Y}_{d,n}(q)$ due to (\ref{form-JM}). }
\end{rem}

\section{The affine Yokonuma--Hecke algebra}\label{sec-aff}

In this section we will see that the Yokonuma--Hecke algebra ${\rm Y}_{d,n}(q)$ is a quotient  of  the {\em affine Yokonuma--Hecke algebra} $\widehat{\rm Y}_{d,n}(q)$, which we define here. We will then study a special class of irreducible representations of $\widehat{\rm Y}_{d,2}(q)$, which will be used in the following sections to determine the spectrum of the Jucys--Murphy elements, and finally all irreducible representations for  ${\rm Y}_{d,n}(q)$.

\subsection{Definition of \,$\widehat{\rm Y}_{d,n}(q)$} 
Let  $d,\,n \in \N$ , $d \geq 1$. Let $q $ be an indeterminate.  The  affine Yokonuma--Hecke algebra, denoted by $\widehat{\rm Y}_{d,n}(q)$ , is a $\C[q,q^{-1}]$-associative algebra generated by the elements
$$\ot_1,\dots,\ot_n,\og_1,\dots,\og_{n-1},X_1^{\pm1}$$
subject to the following relations:
\begin{equation}\label{def-aff1}
\begin{array}{rclcl}
\og_i\og_j & = & \og_j\og_i && \mbox{for all $i,j=1,\dots,n-1$ such that $\vert i-j\vert > 1$,}\\[0.1em]
\og_i\og_{i+1}\og_i & = & \og_{i+1}\og_i\og_{i+1} && \mbox{for  all $i=1,\dots,n-2$,}\\[0.1em]
\ot_i\ot_j & =  & \ot_j\ot_i &&  \mbox{for all $i,j=1,\dots,n$,}\\[0.1em]
\ot_j\og_i & = & \og_i\ot_{s_i(j)} && \mbox{for all $i=1,\dots,n-1$ and $j=1,\dots,n$,}\\[0.1em]
\ot_j^d   & =  &  1 && \mbox{for all $j=1,\dots,n$,}\\[0.2em]
\og_i^2  & = & 1 + (q-q^{-1}) \, \oe_{i} \, \og_i && \mbox{for  all $i=1,\dots,n-1$,}
\end{array}
\end{equation}
where   $s_i$ is the transposition $(i,i+1)$ and
$$\oe_i :=\frac{1}{d}\sum\limits_{s=0}^{d-1}\ot_i^s \ot_{i+1}^{-s}\,,$$
together with the following relations concerning  the 
generators $X_1^{\pm1}$: 
\begin{equation}\label{def-aff2}
\begin{array}{rclcl}
X_1X_1^{-1}& =& X_1^{-1}X_1   \,\,\,\,=\,\,\,\,  1 &&\\[0.1em] 
X_1\,\og_1X_1\og_1 & =  & \og_1X_1\og_1\,X_1  &&\\[0.1em]
X_1\og_i & = & \og_iX_1 && \mbox{for all $i=2,\dots,n-1$,}\\[0.1em]
X_1\ot_j & = & \ot_jX_1 && \mbox{for all $j=1,\dots,n$.}
\end{array}
\end{equation}

Note that the elements $\oe_i$ are idempotents in $\widehat{\rm Y}_{d,n}(q)$ and that the elements $\og_i$ are invertible, with
\begin{equation}
\og_i^{-1} = \og_i - (q- q^{-1})\, \oe_i  \qquad \mbox{for all $i=1,\ldots,n-1$}.
\end{equation}

Similarly to (\ref{edik}), we define, for $i,k=1,2,\ldots, n$,
\begin{equation}\label{edik2}
\oe_{i,k} := \frac{1}{d} \sum_{s=0}^{d-1}\ot_i^s\ot_{k}^{-s}\ .
\end{equation}
Analogues, for the elements $\ot_i$, $\og_j$ and $\oe_{k,l}$, of the relations (\ref{edikrels}) and (\ref{extraone})
are satisfied in $\widehat{\rm Y}_{d,n}(q)$.

We define inductively elements $X_2,\dots,X_n$ in  $\widehat{\rm Y}_{d,n}(q)$ by
\begin{equation}\label{rec-X}
X_{i+1}:=\og_iX_i\og_i\ \ \ \ \text{for $i=1,\dots,n-1$.}
\end{equation}
Note that the 
second 
relation of (\ref{def-aff2}) can  be then rewritten as $X_1X_2=X_2X_1$. 

One motivation for defining the algebra $\widehat{\rm Y}_{d,n}(q)$ is the following: there is a surjective homomorphism $\pi$ from $\widehat{\rm Y}_{d,n}(q)$ onto the Yokonuma--Hecke algebra ${\rm Y}_{d,n}(q)$ given  by
\begin{equation}\label{surj-pi}
\pi(\ot_j)=t_j \ \text{ for } j=1,\dots,n,\quad\ \pi(\og_i)=g_i\ \text{ for } i=1,\dots,n-1,\quad\ \text{and}\quad\ 
\pi(X_1)=1\ .
\end{equation}
The fact that $\pi$ defines an algebra homomorphism from $\widehat{\rm Y}_{d,n}(q)$ to ${\rm Y}_{d,n}(q)$ derives immediately from the comparison of the relations (\ref{def-aff1}) with the defining  relations of ${\rm Y}_{d,n}(q)$, together with the fact that the relations (\ref{def-aff2}) are trivially satisfied if 
$X_1^{\pm 1}$ is replaced by $1$. 
Further, we  have
\begin{equation}\label{surj-X-J}
\pi(X_i)=J_i\ \ \ \ \text{for $i=1,\dots,n$.}
\end{equation}

\subsection{Commutative family of elements in $\widehat{\rm Y}_{d,n}(q)$}
Here we will prove some properties of the elements $X_1,\dots,X_n$ in $\widehat{\rm Y}_{d,n}(q)$, in particular that they form a commutative set. Via the homomorphism $\pi$, these properties are transferred to the Jucys--Murphy elements of the algebra ${\rm Y}_{d,n}(q)$.

\begin{lem}\label{X-g}
For any $i\in\{1,\dots,n\}$, we have 
$$\og_jX_i=X_i\og_j\ \ \ \text{for $j=1,\dots,n-1$ such that $j\neq i-1,i$.}$$
\end{lem}

\begin{proof}
We prove the assertion by induction on $i$. The induction basis, for $X_1$, is part of the defining relations (\ref{def-aff2}). Let $i>1$. Writing  $X_{i+1}= \og_iX_i\og_i$, and using the first relation in (\ref{def-aff1}) together with the induction hypothesis, we obtain immediately that $X_{i+1}$ commutes with $\og_1,\dots,\og_{i-2}$ and with $\og_{i+2},\dots,\og_{n-1}$. So it remains to prove that $X_{i+1}$ commutes with $\og_{i-1}$ for $i>1$.  We write 
$$X_{i+1}\ = \ \og_iX_i\og_i \ = \ \og_i \og_{i-1}X_{i-1}\og_{i-1}\og_i\ . $$ 
Then we have 
$$\og_{i-1}\left(\og_i\og_{i-1}X_{i-1}\og_{i-1}\og_i\right)=\og_i\og_{i-1}\og_i X_{i-1}\og_{i-1}\og_i=\og_i\og_{i-1}X_{i-1}\og_i\og_{i-1}\og_i=\left(\og_i\og_{i-1}X_{i-1}\og_{i-1}\og_i\right)\og_{i-1}\ ,$$
where we have used,  for the first and third equalities,  that $\og_{i-1}\og_i\og_{i-1}=\og_i\og_{i-1}\og_i$, and the induction hypothesis for the second equality.
We deduce that $X_{i+1}$ commutes with $\og_{i-1}$.
\end{proof}

\begin{prop}\label{comm-aff}For all $i,j=1,\ldots,n$, we have
\begin{enumerate}[(a)]  
\item $X_i\ot_j=\ot_jX_i$\,;\smallbreak
\item $X_iX_j=X_jX_i$\,.
\end{enumerate}
\end{prop}

\begin{proof} 
For $(a)$, we will proceed by induction on $i$. For $i=1$, the relations $X_1\ot_j=\ot_jX_1$, $j=1,\dots,n$, are in the defining relations (\ref{def-aff2}) for $\widehat{\rm Y}_{d,n}(q)$. Now let $i\in\{1,\dots,n-1\}$ and assume that $X_i$ commutes with $t_1,\ldots,t_n$. Write $X_{i+1}=  \og_iX_i\og_i$. Using the induction hypothesis and the defining relations of $\widehat{\rm Y}_{d,n}(q)$, the following calculation is straightforward for any $j=1,\dots,n$:
$$\left( \og_iX_i\og_i\right)\ot_j=\og_iX_i\ot_{s_i(j)}\og_i= \og_i\ot_{s_i(j)}X_i\og_i=\ot_j \left(\og_iX_i\og_i\right)\ .$$

For $(b)$, we will show, by induction on $i$, that $X_1,\dots,X_i$ is a commutative set. For $i=1$ there is nothing to prove. Let $i\in\{1,\dots,n-1\}$, and write  $X_{i+1}= \og_iX_i\og_i$. By  Lemma \ref{X-g} and the induction hypothesis, we have that $X_{i+1}$ commutes with $X_1,\dots,X_{i-1}$. So it remains to prove that $X_{i+1}$ commutes with $X_i$. If $i=1$ this statement is the first defining relation in (\ref{def-aff2}). Now assume that $i>1$, and write $X_i= \og_{i-1}X_{i-1}\og_{i-1}$. By Lemma \ref{X-g} and the fact that $X_{i+1}$ commutes with $X_{i-1}$, we conclude that $X_i$ commutes with $X_{i+1}$. 
\end{proof}

\subsection{Commutativity of the Jucys--Murphy elements of ${\rm Y}_{d,n}(q)$}
Using the homomorphism $\pi$, we obtain the following corollaries of Lemma \ref{X-g} and Proposition \ref{comm-aff} respectively:
\begin{cor}\label{J-g}
For any $i\in\{1,\dots,n\}$, we have 
$$g_jJ_i=J_ig_j\ \ \ \text{for $j=1,\dots,n-1$ such that $j\neq i-1,i$.}$$
\end{cor}

\begin{cor}\label{commutativeset}
\label{comm-JM}The elements $t_1,\dots,t_n,J_1,\dots,J_n$, that is, the Jucys--Murphy elements of ${\rm Y}_{d,n}(q)$,  form a commutative set of elements.
\end{cor}

\subsection{Representations of the affine Yokonuma--Hecke algebra $\widehat{\rm Y}_{d,2}(q)$}\label{affine_{d,2}}
\noindent Consider the simplest non-trivial affine Yokonuma--Hecke algebra, that is, $\widehat{\rm Y}_{d,2}(q)$.
This algebra is generated by elements 
$\ot_1,\, \ot_2,\, X_1^{\pm1},\, \og$ 
subject to the following defining relations:
$$
\ot_2 = \og\ot_1\og^{-1}\!\!=\og^{-1}\ot_1\og,\quad \ot_1^d=\ot_2^d=1,\quad 
X_1X_1^{-1}=X_1^{-1}X_1=1, 
\quad \og^2 = 1 + (q-q^{-1}) \, \oe \, \og\ ,$$
where $\oe:=\frac{1}{d}\sum\limits_{s=0}^{d-1}\ot_1^s \ot_2^{-s}$, together with the following commutation relations
 $$ 
\og X_1\og X_1=X_1\og X_1\og\ \quad\ \text{and}\ \quad\  
 \mathsf{A}\mathsf{B}= \mathsf{B}\mathsf{A}\quad\text{for any $ \mathsf{A},\mathsf{B} \in\{\ot_1,\ot_2,X_1
 \}$.}$$
Recall that $X_2$ is defined by $X_2=\og X_1\og$ and that $X_2$ commutes with $\ot_1,\ot_2$ and $X_1$;
also that $\oe$ is an idempotent and that $\og$ is invertible with  
$\og^{-1} = \og + (q^{-1}-q)\,\oe$
(note that $X_2$ is invertible as well). 
Finally, in this particular case, $\oe$ is a central element of $\widehat{\rm Y}_{d,2}(q)$.

We are interested in the irreducible representations of $\C(q)\widehat{\rm Y}_{d,2}(q)$ such that $\ot_1,\,\ot_2,\,X_1,\,X_2$ are diagonalisable operators.
So let ${\bf v}$ be a common eigenvector of $ \ot_1,\, \ot_2,\, X_1,\, X_2$ such that
$$\ot_1({\bf v})=\a {\bf v},\,\ot_2({\bf v})=\b {\bf v},\,X_1({\bf v})=\l {\bf v},\,X_2({\bf v})=\m {\bf v}\,,$$
for some $\a, \b, \l, \m \in \C(q)$, with $\a^d=\b^d=1$,
$\l\neq0$ and $\m\neq 0$. 
Set 
$$\g:=\frac{1}{d}\sum_{s=0}^{d-1}\a^s \b^{-s}.$$
We have $\oe({\bf v})=\g {\bf v}$. Since $\oe$ is an idempotent, we must have $\g=0$ or $\g=1$.  In fact,  $\g=1$ if and only if $\a=\b$.

First assume that $\og( {\bf v})$ is proportional to ${\bf v}$, that is, $\og( {\bf v}) = \r {\bf v}$ for some {$\r \in \C(q)\setminus \{0\}$.
Since $\ot_2=\og\ot_1\og^{-1}$, we must have $\a=\b$, so we have $\g=1$. Moreover, $X_2({\bf v})= \og X_1\og ({\bf v})= \l \r^2{\bf v}$.
Finally, the relation
$$ \og \left( \og({\bf v}) \right) = {\bf v} + (q-q^{-1}) \,  \oe\, \og({\bf v})$$
implies that $\r$ satisfies the quadratic equation
$$ \r^2-(q-q^{-1})\r-1=0,$$
whence we deduce that $\r=\epsilon q^{\epsilon}$, where $\epsilon\in\{-1,1\}$.
Thus, the vector ${\bf v}$ spans a one-dimensional representation of $\C(q)\widehat{\rm Y}_{d,2}(q)$, and we have
\[\ot_1({\bf v})=\a {\bf v},\ \ \ot_2({\bf v})=\a {\bf v},\ \ X_1({\bf v})=\l {\bf v},\ \ X_2({\bf v})=\l q^{2\epsilon}\,{\bf v},\ \ \og({\bf v})=\epsilon q^{\epsilon}{\bf v}\,,\]
with $\epsilon\in\{-1,1\}$.

Now assume that $\og( {\bf v})$ is not proportional to ${\bf v}$. 
We will calculate the actions of $ \ot_1$, $\ot_2$, $X_1$, $X_2$, $\og$ on $\og( {\bf v})$, using the quadratic relation $\og^2=1 + (q-q^{-1}) \, \oe \, \og$, together with
$$\ot_1 \og=\og \ot_2, \qquad \ot_2 \og=\og \ot_1,\qquad X_1 \og= \og^{-1}X_2\ \ \ \ \ \text{and}\ \ \ \ \ X_2 \og=\og X_1 \og^2\,.$$     
We obtain that the action of $\C(q)\widehat{\rm Y}_{d,2}(q)$ closes on the  vector space spanned by ${\bf v}$ and $\og( {\bf v})$. The explicit formulas are given by
$$
\ot_1 \og({\bf v})= \og \ot_2 ({\bf v}) = \b \og( {\bf v}) \,\,\,\text{ and }\,\,\, \ot_2 \og({\bf v})= \og \ot_1 ({\bf v}) = \a \og( {\bf v}).
$$
Thus, $\oe\og({\bf v})=\g \og({\bf v})$.
Now,
$$X_1 \og({\bf v}) =  \og^{-1} X_2({\bf v})=
 \m\, \og^{-1}({\bf v})
=\m\, (q^{-1}\!-q)\,\g\,{\bf v}+\m\,\og({\bf v}),$$
and
$$X_2 \og({\bf v})=\og X_1\og^2({\bf v})  =\og X_1\bigl({\bf v}+(q-q^{-1})\g \og({\bf v})\bigr)= \og X_1({\bf v})+ (q-q^{-1})\g X_2({\bf v})= \m (q-q^{-1})\,\g\,{\bf v}+\l \, \og({\bf v}).$$
Finally,
$$ \og \left( \og({\bf v}) \right) = \og^2({\bf v}) =  {\bf v} +(q-q^{-1}) \, \g \, \og({\bf v}).$$

\vskip .4cm
As a conclusion, we obtain the following classification of irreducible representations of $\C(q)\widehat{\rm Y}_{d,2}(q)$ with $\ot_1,\,\ot_2,\,X_1,\,X_2$ diagonalisable:

\begin{enumerate}[(1)]
\item One-dimensional representations, given by:
$$\ot_1\mapsto\a\,,\ \ \ \ \ot_2\mapsto\a\,,\ \ \ \ X_1\mapsto\l\,,\ \ \ \ X_2\mapsto\l q^{2\epsilon}\,,\ \ \ \ \og\mapsto\epsilon q^{\epsilon},$$
where $\epsilon\in\{-1,1\}$, $\a^d=1$ and $\l\neq0$.

\vskip .2cm
\item Two-dimensional representations with the central element $\oe$ acting as the identity matrix, given by:
$$
\ot_1\mapsto\left(\begin{array}{cc}\a & 0\\0 & \a \end{array}\right)\!,\ \ \ot_2\mapsto\left(\begin{array}{cc}\a & 0\\0 & \a \end{array}\right)\!,\ \ X_1\mapsto\left(\begin{array}{cc}\l & -\m (q-q^{-1})\\0 & \m  \end{array}\right)\!,\ \ X_2\mapsto\left(\begin{array}{cc}\m & \m (q-q^{-1})\\0 & \l  \end{array}\right)\!,$$
$$\og\mapsto\left(\begin{array}{cc}0 & 1\\1 & q-q^{-1} \end{array}\right),\ \ $$ 
 where $\a^d=1$,
 $\l,\m\neq0$ 
  and  $\m\neq\l$ (in order for $X_1,X_2$ to be diagonalisable). In order  for these representations to be irreducible, we must have $\m\neq\l q^{\pm2}$.
Then, for the basis $$\left\{{\bf v}\,,\,\frac{\m-\l}{q\m -q^{-1}\l}\left({\og({\bf v})}-\frac{\m\,(q-q^{-1})}{\m-\l}{\bf v}\right)\right\},$$ the action of the generators becomes
$$\ot_1\mapsto\left(\begin{array}{cc}\a & 0\\0 & \a \end{array}\right),\ \ \ot_2\mapsto\left(\begin{array}{cc}\a & 0\\0 & \a \end{array}\right),\ \ X_1\mapsto\left(\begin{array}{cc}\l & 0\\0 & \m \end{array}\right),\ \ X_2\mapsto\left(\begin{array}{cc}\m & 0\\0 & \l  \end{array}\right),$$
$$\og\mapsto \frac{1}{\m -\l}
\left(\begin{array}{cc}\displaystyle{\m(q-q^{-1})} & \displaystyle {-\bigl(q\l -q^{-1}\m\bigr)}\\[1.2em] \bigl(q\m-q^{-1}\l\bigr) & \displaystyle{-\l (q-q^{-1})}\end{array}\right).\ \ $$  

\vskip .2cm
\item Two-dimensional representations with the central element $\oe$ acting as $0$, given by:
$$\ot_1\mapsto\left(\begin{array}{cc}\a & 0\\0 & \b \end{array}\right),\ \ \ot_2\mapsto\left(\begin{array}{cc}\b & 0\\0 & \a \end{array}\right),\ \ X_1\mapsto\left(\begin{array}{cc}\l & 0\\0 & \m  \end{array}\right),\ \ X_2\mapsto\left(\begin{array}{cc}\m & 0\\0 & \l  \end{array}\right),\ \ \og\mapsto\left(\begin{array}{cc}0 & 1\\1 & 0 \end{array}\right),\ \  $$
where $\a^d=\b^d=1$,
$\a\neq\b$ and $\l,\m\neq0$
. All these representations are  irreducible.
\end{enumerate}

\section{Spectrum of the Jucys--Murphy elements and standard $d$-tableaux}\label{sec-spe}

In this section we will use the representations of the affine Yokonuma--Hecke algebra $\widehat{\rm Y}_{d,2}(q)$  constructed in \S \ref{affine_{d,2}} to describe the spectrum of the Jucys--Murphy elements of ${\rm Y}_{d,n}(q)$. We will then provide a connection between the spectrum and the combinatorics of standard $d$-tableaux.

\subsection{Spectrum of the Jucys--Murphy elements}
We shall study the spectrum of the Jucys--Murphy elements for the representations of $\C(q){\rm Y}_{d,n}(q)$ that verify the following two conditions: 
\begin{enumerate}[(i)]
\item the Jucys--Murphy elements  $t_1,\dots,t_n,J_1,\dots,J_n$ are represented by diagonalisable operators; 
\item for any $i=1,\dots,n-1$, 
the action of the subalgebra of $\C(q){\rm Y}_{d,n}(q)$ generated by $t_i$, $t_{i+1}$, $J_i$, $J_{i+1}$ and $g_i$ is completely reducible. 
\end{enumerate}
In Section \ref{sec-rep}, we shall see that all irreducible representations of $\C(q){\rm Y}_{d,n}(q)$ satisfy these two properties.

\begin{defn}{\rm We say that the  $2 \times n$  array
\begin{equation}\label{spec-ele}\Lambda=\left(\begin{array}{ccc}\a^{(\Lambda)}_1&,\, \dots\, ,&\a^{(\Lambda)}_n\\[.5em]
\l^{(\Lambda)}_1&,\, \dots\, ,&\l^{(\Lambda)}_n\end{array}\right)\end{equation}
belongs to the {\em spectrum} of  $t_1,\dots,t_n,J_1,\dots,J_n$ 
if there exists a representation $V_\Lambda$ of  $\C(q){\rm Y}_{d,n}(q)$ satisfying conditions (i) and (ii), and  a vector ${\bf v}_{\Lambda} \in V_\Lambda$ 
such that
\begin{center}
 $t_i({\bf v}_{\Lambda})=\a^{(\Lambda)}_i{\bf v}_{\Lambda}$ and $J_i({\bf v}_{\Lambda})=\l^{(\Lambda)}_i{\bf v}_{\Lambda}$ for $i=1,\dots,n$. 
 \end{center}
In this case, we say that the vector ${\bf v}_{\Lambda}$ is an {\em admissible vector} for $\Lambda$. 
 We denote by $\spec$ the spectrum of  the Jucys--Murphy elements of ${\rm Y}_{d,n}(q)$.}
\end{defn}

\begin{rem}{\rm 
By  Corollary \ref{J-g}, we have that the action of $g_j$, $j=1,\dots,n-1$, on 
an admissible vector 
${\bf v}_\Lambda$ is ``local" in the sense that $g_j({\bf v}_{\Lambda})$ is a linear combination of 
admissible vectors 
${\bf v}_{\Lambda'}$ 
for arrays $\Lambda'$ 
such that 
\begin{center}
${{\a^{(\Lambda')}_i=\a^{(\Lambda)}_i}^{\phantom{A}}}^{\phantom{A}}$\hspace{-0.55cm}\,\,\, and\,\,\, $\l^{(\Lambda')}_i=\l^{(\Lambda)}_i$ \,\,\,for $i\neq j,j+1$. 
\end{center}}
\end{rem} 

\vskip .2cm 
For any $i\in\{1,\dots,n-1\}$, the elements $t_i$, $t_{i+1}$, $J_i$, $J_{i+1}$ and $g_i$ realise the simplest non-trivial affine Yokonuma--Hecke algebra $\widehat{\rm Y}_{d,2}(q)$.
From the study of the representations of  $\C(q)\widehat{\rm Y}_{d,2}(q)$ made in  \S \ref{affine_{d,2}}, we obtain the following information on the set $\spec$.
\begin{prop} \label{prop-spec}
Let $\Lambda=\left(\begin{array}{ccccc}\a_1&,\, \dots\, ,&\a_i\ ,\ \a_{i+1}&,\, \dots\, ,&\a_n\\[.3em]
\l_1&,\, \dots\, ,&\l_i\ ,\ \l_{i+1}&,\, \dots\, ,&\l_n\end{array}\right)\in 
\spec$, and let ${\bf v}_{\Lambda}$ be 
an admissible vector for $\Lambda$. 
Then

\begin{itemize}
\item[(a)] We have $\a_i^d=1$ for all $i=1,\dots,n$; if $\a_{i+1}=\a_{i}$, then $\l_{i+1}\neq \l_i$. 
\vskip .2cm
\item[(b)] If $\a_{i+1}=\a_i$ and $\l_{i+1}=\l_iq^{2\epsilon}$, where $\epsilon=\pm 1$, then $g_i({\bf v}_{\Lambda})=\epsilon q^{\epsilon}{\bf v}_{\Lambda}$.
\vskip .2cm
\item[(c)] If $\a_{i+1}=\a_i$ and $\l_{i+1}\neq \l_iq^{\pm2}$, then 
\[\Lambda'=\left(\begin{array}{cccccc}\a_1&,\, \dots\, ,&\a_{i+1}\ ,& \a_i &,\, \dots\, ,&\a_n\\[.3em]
\l_1&,\, \dots\, ,&\l_{i+1}\ ,&\l_i&,\, \dots\, ,&\l_n\end{array}\right)\in \spec\ .\] 
Moreover, the vector $g_i({\bf v}_{\Lambda})-\frac{\l_{i+1}(q-q^{-1})}{\l_{i+1}-\l_i}{\bf v}_{\Lambda}$ 
is admissible for 
the array $\Lambda'$.
\vskip .2cm
\item[(d)] If $\a_{i+1}\neq \a_i$, then  
\[\Lambda'=\left(\begin{array}{cccccc}\a_1&,\, \dots\, ,&\a_{i+1}\ ,& \a_i&,\, \dots\, ,&\a_n\\[.3em]
\l_1&,\, \dots\, ,&\l_{i+1}\ ,&  \l_i&,\, \dots\, ,&\l_n\end{array}\right)\in \spec\ .\] 
Moreover, the vector $g_i({\bf v}_{\Lambda})$ 
is admissible for 
the array $\Lambda'$.
\end{itemize}
\end{prop}

\subsection{Content arrays}
We define the following set of $2\times n$ arrays with entries in $\C(q)$ which will turn out, on the one hand, to contain the set $\spec$ and, on the other hand, to be in bijection with the set of standard $d$-tableaux of size $n$.
\begin{defn} \label{def-cont}{\rm
A {\em content array} is a $2\times n$ array of numbers with entries in $\C(q)$, $$\left(\begin{array}{ccc}\a_1&,\, \dots\, ,&\a_n\\[.3em]
\l_1&,\, \dots\, ,&\l_n\end{array}\right),$$ satisfying the following conditions: 
\begin{itemize}
\item[(1)] We have $\l_1=1$, and $\a_i^d=1$ for all $i=1,\dots,n$. 
\vskip .2cm
\item[(2)] 
If  $\l_j\neq 1$ for some $j>1$, then there exists $i<j$ such that $\a_i=\a_j$ and
$\l_i\in\{\l_jq^{-2},\l_jq^2\}$.
\vskip .2cm
\item[(3)] If $\a_j=\a_k$ and $\l_j=\l_k$ for
 $j<k$, then $k-j \geq 3$ and there exist $j+1 \leq i_1, \, i_2 \leq k-1$ such that $ \a_{i_1}=\a_{i_2}=\a_j$, 
$\l_{i_1}=\l_jq^{-2}$ and $\l_{i_2}=\l_jq^2$.
\end{itemize}
We denote by $\cont$ the set of content arrays.
}
\end{defn}

With the help of the Proposition \ref{prop-spec}, we can now prove the following result.

\begin{prop} \label{spec-cont}
We have $\spec\subseteq \cont$.
\end{prop}

\begin{proof}
Following Proposition \ref{prop-spec}(a), for all $\Lambda=\left(\begin{array}{ccc}\a_1&,\, \dots\, ,&\a_n\\[.3em]
\l_1&,\, \dots\, ,&\l_n\end{array}\right)\in 
\spec$,  we have $\a_i^d=1$ for all $i=1,\dots,n$. Moreover, since $J_1=1$, we have
$\l_1=1$. Thus, condition (1) of Definition \ref{def-cont} is satisfied.

We will prove that condition (2) of Definition \ref{def-cont} holds for all $\Lambda\in \spec$ by induction on $j$.
First, let $j=2$ and assume that $\l_2\neq 1$. If $\a_1 \neq \a_2$, then, by Proposition \ref{prop-spec}(d), we have that
\[\Lambda'=\left(\begin{array}{ccccc}\a_2,&\,\a_{1}\ ,& \, \dots\, ,&\a_n\\[.3em]
\l_2,&\, \l_1,&\, \dots\, ,&\l_n\end{array}\right)
=\left(\begin{array}{ccc}\a_1',&\, \dots\, ,&\a_n'\\[.3em]
\l_1',&\, \dots\, ,&\l_n'\end{array}\right)
\in \spec\ ,\]
which contradicts the fact that $\l_1'=1$. So we must have $\a_1 = \a_2$. Now if $\l_2 \neq q^{\pm2}$, then Proposition \ref{prop-spec}(c) yields again that $\Lambda' \in \spec$.
So we deduce that $\l_2= q^{\pm2}$ and condition (2) is satisfied.

Similarly, if $j>2$ and $\l_j\neq 1$, then, unless $\a_{j-1} = \a_j$ and $\l_{j-1}=  \l_jq^{\pm2}$, Proposition \ref{prop-spec} implies that 
\[\Lambda'=\left(\begin{array}{cccccc}\a_1&,\, \dots\, ,&\a_{j}\ ,& \a_{j-1}&,\, \dots\, ,&\a_n\\[.3em]
\l_1&,\, \dots\, ,&\l_{j} ,&  \l_{j-1}&,\, \dots\, ,&\l_n\end{array}\right)\in \spec\ ,\]
and induction hypothesis yields the desired result. In every case, condition (2) holds.

Now, let $1\leq j<k\leq n$ such that  $\a_j=\a_k$ and $\l_j=\l_k$. We will prove that condition (3) of Definition \ref{def-cont} holds for all $\Lambda\in \spec$. 
We will first show that  $k-j\geq3$. 
 Note that, due to Proposition \ref{prop-spec}(a), we must have $k-j>1$. If $k-j=2$, then, unless $\a_{j+1}=\a_j$ and $\l_{j+1}=\l_j
 q^{\pm2}$, Proposition \ref{prop-spec}(c)--(d) 
  implies that 
 \[\Lambda'=\left(\begin{array}{ccccccc}\a_1&,\, \dots\, ,&\a_{j+1}\ ,& \a_j,& \a_k&\, \dots\, ,&\a_n\\[.3em]
\l_1&,\, \dots\, ,&\l_{j+1} ,&  \l_j,& \l_k&,\, \dots\, ,&\l_n\end{array}\right)\in \spec\ ,\]
which contradicts Proposition \ref{prop-spec}(a). However, if $\a_{j+1}=\a_j$ and $\l_{j+1}=\l_j
 q^{2\epsilon}$, with $\epsilon=\pm1$, then Proposition \ref{prop-spec}(b) implies that $g_j({\bf v}_{\Lambda})=\epsilon q^{\epsilon}{\bf v}_{\Lambda}$ and $g_{j+1}({\bf v}_{\Lambda})=-\epsilon q^{-\epsilon}{\bf v}_{\Lambda}$, which contradicts the relation $g_jg_{j+1}g_j=g_{j+1}g_jg_{j+1}$. Thus, we must have $k-j \geq 3$.
 
We will now prove that there exist $j+1 \leq i_1, \, i_2 \leq k-1$ such that $ \a_{i_1}=\a_{i_2}=\a_j$, 
$\l_{i_1}=\l_jq^{-2}$ and $\l_{i_2}=\l_jq^2$ by induction on the difference $k-j$. 
First assume that $k-j=3$. 
Unless $\a_j=\a_{j+1}=\a_{k-1}$, $\l_{j+1}=\l_jq^{\pm 2}$ and $\l_{k-1}=\l_kq^{\pm 2}$,
we can use Proposition \ref{prop-spec}(c)--(d) 
to go back to the case $k-j=2$. Moreover, by Proposition \ref{prop-spec}(a), if $\a_{j+1}=\a_{k-1}$, we must have  $\l_{j+1}\neq\l_{k-1}$.
Thus, by taking $\{i_1,i_2\}=\{j+1,k-1\}$, condition (3) is satisfied.

Finally, assume that $k-j >3$. Now, unless
$\a_j=\a_{j+1}=\a_{k-1}$,
$\l_{j+1}=\l_jq^{\pm2}$ and $\l_{k-1}=\l_kq^{\pm2}$, Proposition 
\ref{prop-spec}(c)-(d) 
implies that
 \[\Lambda'=\left(\begin{array}{llllllll}\a_1&,\, \dots\, ,&\a_{j+1},& \a_j,    &\, \dots\, , & \a_k ,&\, \dots\, ,&\a_n\\[.3em]
\l_1&,\, \dots\, ,&\l_{j+1} ,&  \l_j,& \, \dots\, ,& \l_k,&  \, \dots\, ,&\l_n\end{array}\right)\in \spec\ ,\] 
or
 \[\Lambda''=\left(\begin{array}{llllllll}\a_1&,\, \dots\, ,&\a_j,    &\, \dots\, , & \a_{k},& \a_{k-1} ,&\, \dots\, ,&\a_n\\[.3em]
\l_1&,\, \dots\, ,&  \l_j,& \, \dots\, ,& \l_{k} ,&\l_{k-1},&  \, \dots\, ,&\l_n\end{array}\right)\in \spec\ ,\] 
and induction hypothesis yields the desired result.
If $\a_j=\a_{j+1}=\a_{k-1}$,
$\l_{j+1}=\l_jq^{\pm2}$ and $\l_{k-1}=\l_kq^{\pm2}$, then we have to distinguish two cases. If $\l_{j+1}=\l_{k-1}=\l_jq^{2\epsilon}$, with $\epsilon=\pm1$, then the induction hypothesis implies  that there exists $j+1<l<k-1$ such that $\a_l=\a_{j+1}=\a_j$ and $\l_l=\l_{j+1}q^{-2\epsilon}=\l_j$. Thus we can replace $(j,k)$ by $(j,l)$, and the induction hypothesis yields the rest.
If $\l_{j+1} \neq \l_{k-1}$, then by taking $\{i_1,i_2\}=\{j+1,k-1\}$, we can have 
$\a_{i_1}=\a_{i_2}=\a_j$, 
$\l_{i_1}=\l_jq^{-2}$ and $\l_{i_2}=\l_jq^2$, 
and so condition (3) is satisfied.
\end{proof}

\subsection{Standard $d$-tableaux and spectrum of the Jucys--Murphy elements}
\noindent In this subsection  we will interpret the $2\times n$ arrays  belonging to $\cont$ as standard $d$-tableaux. First, we will need to recall some basic notions and facts about the combinatorics of $d$-partitions.

\subsubsection{Combinatorics of $d$-partitions}
Let $\lambda\vdash n$ be a partition of $n$, that is, $\lambda=(\lambda_1,\dots,\lambda_k)$ is a family of  positive integers such that $\lambda_1\geq\lambda_2\geq\dots\geq\lambda_k \geq 1$ and $|\lambda|:=\lambda_1+\dots+\lambda_k=n$. We shall also say that $\lambda$ is a partition {\em of size} $n$. 

We identify partitions with their Young diagrams: the Young diagram of $\lambda$ is a left-justified array of $k$ rows such that
the $j$-th row contains  $\lambda_j$ {\em nodes } for all $j=1,\dots,k$. We write $\theta=(x,y)$ for the node in row $x$ and column $y$. A node $\theta \in \lambda$ is called {\it removable} if the set of nodes obtained from $\lambda$ by removing $\theta$ is still a partition. A node $\theta' \notin \lambda$ is called {\it  addable} if the set of nodes obtained from $\lambda$ by adding $\theta'$ is still a partition. 

The {\em conjugate partition} $\lambda'=(\lambda'_1,\dots,\lambda'_l)$  of $\lambda$ is defined by $\lambda'_{j}:={\#\{i\,|\,1 \leq i \leq k \text{ such that } \lambda_i\geq j\}}$ 
and $l:=\text{max}\{j\,|\,\lambda_j'\neq 0\}$. 
The Young diagram of $\lambda'$ is the transpose of the Young diagram of $\lambda$. A node $\theta=(x,y) \in \lambda$ if and only if $(y,x) \in \lambda'$.

A $d$-partition $\blambda$, or a Young $d$-diagram, of size $n$ is a $d$-tuple of partitions such that the total number of nodes in the associated Young diagrams is equal to $n$. That is, we have $\blambda=(\blambda^{(1)},\dots,\blambda^{(d)})$ with $\blambda^{(1)},\dots,\blambda^{(d)}$ usual partitions such that $|\blambda^{(1)}|+\dots+|\blambda^{(d)}|=n$.

A pair $\btheta=(\theta,k)$ consisting of a node $\theta$ and an integer $k\in\{1,\dots,d\}$ is called a $d$-node. The integer $k$ is called the \emph{position} of $\btheta$. A $d$-partition is then a set of $d$-nodes such that the subset consisting of the $d$-nodes having position $k$ forms a usual partition, for any $k\in\{1,\dots,d\}$.

Let ${\blambda}=(\blambda^{(1)},\dots,\blambda^{(d)})$ be a $d$-partition. A $d$-node $\btheta=(\theta,k)\in\blambda$ is called removable from $\blambda$ if the node $\theta$ is removable from $\lambda^{(k)}$. A $d$-node $\btheta'=(\theta',k')\notin\blambda$ is called addable to $\blambda$ if the node $\theta'$ is addable to  $\lambda^{(k')}$. The set of $d$-nodes removable from $\blambda$ is denoted by ${\mathcal{E}}_-(\blambda)$ and the set of $d$-nodes addable to $\blambda$ is denoted by ${\mathcal{E}}_+(\blambda)$. For example, the removable/addable $3$-nodes (marked with $-/+$) for the $3$-partition $\left(\Box\!\Box,\varnothing,\Box\right)$ 
are:
\[\left(\begin{array}{l}\fbox{$\phantom{-}$}\fbox{$-$}\fbox{$+$}\\ \fbox{$+$}\end{array}
\, ,\, \begin{array}{l}\fbox{$+$}\\ \phantom{\fbox{$-$}}\end{array}\, ,\,\begin{array}{l}\fbox{$-$}\fbox{$+$}\\ \fbox{$+$}\end{array}\right). \]
We will write $\blambda \setminus \{\btheta\}$ for the $d$-partition obtained by removing the $d$-node  $\btheta \in {\mathcal{E}}_-(\blambda)$ from $\blambda$,
and
$\blambda \cup \{\btheta'\}$ for the $d$-partition obtained by adding the $d$-node  $\btheta' \in {\mathcal{E}}_+(\blambda)$ to $\blambda$.

For a $d$-node $\btheta$ lying in the line $x$ and the column $y$ of the $k$-th diagram of $\blambda$ (that is, $\btheta=(x,y,k)$), we define $\pos(\btheta):=k$ and $\cc(\btheta):=q^{2(y-x)}$. The number $\pos(\btheta)$ is the position of $\btheta$ and the number $\cc(\btheta)$ is called the \emph{(quantum) content} of $\btheta$. 

If two $d$-nodes $\btheta, \btheta' \in \blambda$ satisfy $\pos(\btheta)=\pos(\btheta')$ and $\cc(\btheta)=\cc(\btheta')$,  they lie in the same diagonal of the Young diagram of $\blambda^{(\pos(\btheta))}$. If, moreover,
$\cc(\btheta)=\cc(\btheta')=1$, then $\btheta, \btheta'$ lie in the main diagonal of $\blambda^{(\pos(\btheta))}$.

Finally, we define the {\em hook length} $\mathrm{hl}({\btheta})$ of  the $d$-node $\btheta=(x,y,k)$ to be the integer 
\begin{equation}\label{hl}\mathrm{hl}({\btheta}):=\blambda^{(k)}_x-x+\blambda_y^{(k)'}-y+1\ .\end{equation} 

\subsubsection{Standard $d$-tableaux}
Let $\blambda=(\blambda^{(1)},\ldots,\blambda^{(d)})$ be a $d$-partition of $n$. A {\em $d$-tableau of shape $\blambda$
} is a bijection between the set $\{1,\dots,n\}$ and the set of $d$-nodes in $\blambda$. In other words, a $d$-tableau of shape $\blambda$ is obtained by placing the numbers $1,\dots,n$ in the $d$-nodes of $\blambda$. 
The \emph{size} of a $d$-tableau of shape $\blambda$ is $n$, that is, the size of $\blambda$.  
 A $d$-tableau is {\em standard} if its entries  increase along any row and down 
 any column of every diagram in $\blambda$. For $d=1$, a standard $1$-tableau is a usual standard tableau.

For a 
$d$-tableau ${\mathcal{T}}$, we denote respectively by $\pos({\mathcal{T}}|i)$ and $\cc({\mathcal{T}}|i)$  the position and the quantum content  of the $d$-node with the number $i$ in it. For example, for the standard $3$-tableau ${\mathcal{T}}^{^{\phantom{A}}}\!\!\!\!={\textrm{$\left(
\,\fbox{\scriptsize{$1$}}\fbox{\scriptsize{$3$}}\, ,\,\varnothing\, ,\,\fbox{\scriptsize{$2$}}\,\right)$}}$ of size $3$, we have
\[\pos({\mathcal{T}}|1)=1\,,\ \ \pos({\mathcal{T}}|2)=3\,,\ \ \pos({\mathcal{T}}|3)=1\ \ \ \ \ \text{and}\ \ \ \ \  \cc({\mathcal{T}}|1)=1\,,\ \ \cc({\mathcal{T}}|2)=1\,,\ \ \cc({\mathcal{T}}|3)=q^2\,.\]

From now on, we will denote by $\mathrm{STab}_d(n)$ the set of all standard $d$-tableaux of size $n$ (of any shape).

\begin{prop}\label{cont-tab}
The set $\mathrm{STab}_d(n)$ is in bijection with the set $\cont$.
\end{prop}

\begin{proof}
Let $\{\xi_1,\dots,\xi_d\}$ be the set of all $d$-th roots of unity (ordered arbitrarily), and let $\mathcal{A}^{2 \times n}$ be the set of all $2 \times n$ arrays with entries in $\C(q)$. 
We  construct a map $f: \mathrm{STab}_d(n) \rightarrow \mathcal{A}^{2 \times n}$ such that
$$f \left( {\mathcal{T}} \right) =  \left(\begin{array}{ccc}\a_1&,\, \dots\, ,&\a_n\\[.3em]
\l_1&,\, \dots\, ,&\l_n\end{array}\right), $$
where, for all $i=1,\ldots,n$, $\a_i := \xi_{\pos({\mathcal{T}}|i)}$ and $\l_i: = \cc({\mathcal{T}}|i)$.

Let ${\mathcal{T}}, {\mathcal{T}}' \in \mathrm{STab}_d(n)$ such that $f( {\mathcal{T}})=f( {\mathcal{T}}' )$.
 Then, for all $i=1,\ldots,n$, we have
\[\pos({\mathcal{T}}|i)=\pos({\mathcal{T}'}|i)\ \ \ \ \ \text{and}\ \ \ \ \  \cc({\mathcal{T}}|i)= \cc({\mathcal{T}}'|i)\,.\]
So ${\mathcal{T}}$ and ${\mathcal{T}}'$ must have the same shape, and moreover the $d$-tableau ${\mathcal{T}}'$ is obtained from ${\mathcal{T}}$ by permuting the entries inside each diagonal of each diagram. 
However, since ${\mathcal{T}}, {\mathcal{T}}'$ are standard $d$-tableaux, the entries increase across each diagonal and therefore we cannot have a non-trivial permutation of the entries.
We conclude that ${\mathcal{T}}={\mathcal{T}}'$, and thus $f$ is injective. 

In order to obtain the desired bijection, it remains to show that ${\rm Im}f=\cont$. 
Let
$$\left(\begin{array}{ccc}\a_1&,\, \dots\, ,&\a_n\\[.3em]
\l_1&,\, \dots\, ,&\l_n\end{array}\right)= f \left( {\mathcal{T}} \right) $$
for some ${\mathcal{T}} \in \mathrm{STab}_d(n)$ of shape $\blambda=(\blambda^{(1)},\dots,\blambda^{(d)})$.
By definition of $f$, we have $\a_i^d=1$ for all $i=1,\ldots,n$. Moreover, we always have $\cc({\mathcal{T}}|1)=1$, and so $\l_1=1$.

Now assume that $\l_j = \cc({\mathcal{T}}|j) \neq 1$ for some $j=2,\ldots,n$. Then the number $j$ is in a node of the diagram 
of $\blambda^{(\pos({\mathcal{T}}|j))}$, not on the main diagonal (which contains the nodes with quantum {content $1$}). So there is surely a node left to it or above it. If there is a node left to it, then the number $i$ in it will satisfy $i < j$, $\pos({\mathcal{T}}|i)=\pos({\mathcal{T}}|j)$ and $\cc({\mathcal{T}}|i)=\cc({\mathcal{T}}|j)q^{-2}$ (since the $d$-tableau ${\mathcal{T}}$ is standard).
Similarly, if there is a node above it, then the number $i$ in it will satisfy $i < j$, $\pos({\mathcal{T}}|i)=\pos({\mathcal{T}}|j)$ and $\cc({\mathcal{T}}|i)=\cc({\mathcal{T}}|j)q^2$.
Thus, there exists $i<j$ such that $\a_i=\a_j$ and
$\l_i\in\{\l_jq^{-2},\l_jq^2\}$.

Finally, assume that $\a_j=\a_k$ and $\l_j=\l_k$ for some $1 \leq j < k \leq n$.
Then $\pos({\mathcal{T}}|j)=\pos({\mathcal{T}}|k)$ and the numbers $j$ and $k$ are in nodes $\btheta$ and $\btheta'$ respectively belonging to the same diagonal of the diagram of  $\blambda^{(\pos({\mathcal{T}}|j))}$. Since $\mathcal{T}$ is a standard $d$-tableau and $j<k$,
there must be a node just under $\btheta$, with a number $i_1$ in it, and a node just on the right of $\btheta$, with an entry $i_2$ in it; moreover,  we have $j+1 \leq i_1,i_2 \leq k-1$. Therefore, $k-j\geq3$, and there exist $j+1 \leq i_1, \, i_2 \leq k-1$ such that $\a_{i_1}=\a_{i_2}=\a_j$, 
$\l_{i_1}=\l_jq^{-2}$ and $\l_{i_2}=\l_jq^2$.
We conclude that ${\rm Im}f \subseteq \cont$.

Now let
$$\Lambda=\left(\begin{array}{ccc}\a_1&,\, \dots\, ,&\a_n\\[.3em]
\l_1&,\, \dots\, ,&\l_n\end{array}\right) \in \cont.$$
We will  show that $\Lambda \in {\rm Im}f $ by constructing a standard $d$-tableau $\mathcal{T}$ of size $n$ such that $f(\mathcal{T})=\Lambda$.
In order to do this, we will start with the empty $d$-tableau and in each step we will add a $d$-node with the entry $i$ in it so that
$$\a_i = \xi_{\pos({\mathcal{T}}|i)}\ \ \ \ \ \text{and}\ \ \ \ \  \l_i=\cc({\mathcal{T}}|i).$$
The $d$-node with the entry $i$ will be added in the first non-occupied position of the diagonal determined by $\pos({\mathcal{T}}|i)$ and $\cc({\mathcal{T}}|i)$.
We will then show that, 
after  each step, the result of the construction is a standard $d$-tableau.

We will proceed by induction on $i$. The $d$-node with the entry $1$ will be added in the position $\pos({\mathcal{T}}|1)$ (determined by $\a_1$).
Now assume that we have constructed a standard $d$-tableau $\mathcal{T}'$ of size $i-1$ such that
$$\a_j = \xi_{\pos({\mathcal{T}}'|j)} \ \ \ \ \ \text{and}\ \ \ \ \  \l_j = \cc({\mathcal{T}}'|j)$$
for all $j=1,\ldots,i-1$.
Let $\blambda=(\blambda^{(1)},\dots,\blambda^{(d)})$ be the shape of $\mathcal{T}'$.
Then the $d$-node with the entry $i$ has to be added to the partition $\blambda^{(\pos(\cT|i))}$ in
the first non-occupied position of the diagonal with quantum content $\l_i$. 

First assume that $\l_j \neq \l_i$ for all $j=1,\ldots,i-1$ such that $\a_j=\a_i$. Then adding the node with the entry $i$ to $\blambda^{(\pos(\cT'|i))}$ will construct a new diagonal.
If $\l_i =1$ then, by induction hypothesis, 
the partition $\blambda^{(\pos(\cT'|i))}$ 
must be 
empty and the node with the entry $i$ is added in the first place of the main diagonal. 
If $\l_i \neq 1$, then by property (2) of Definition \ref{def-cont}, there exists $k < i$ such that $\a_k=\a_i$ and $\l_k  \in \{\l_i q^{-2},\l_i q^{2}\}$.  
As $\l_j \neq \l_i$ for all $j=1,\ldots,i-1$ such that $\a_j=\a_i$, there is a unique $k$ satisfying these conditions since, by induction hypothesis, $\cT'$ is a standard $d$-tableau. If $\l_k=\l_i q^{-2}$, then the node with the entry $i$ is added just to the right of the node containing $k$.
If $\l_k=\l_i q^2$, then the node with the entry $i$ is added just under the node containing $k$. In every case, since the $d$-tableau $\mathcal{T}'$ is standard, the result is a standard $d$-tableau of size $i$.

Otherwise, let $j$ be the largest integer such that $j<i$, 
$\a_j=\a_i$ and $\l_j=\l_i$. By property (3) of Definition \ref{def-cont}, 
 there exist $j+1 \leq i_1, \, i_2 \leq i-1$ such that $\a_j=\a_{i_1}=\a_{i_2}=\a_i$, 
$\l_{i_1}=\l_jq^{-2}$ and $\l_{i_2}=\l_jq^2$. Since the $d$-tableau $\mathcal{T}'$ is standard, this implies that the $d$-node $(x,y,\pos({\mathcal{T}}|i))$ containing $j$ has a node under it and a node on its right.  We can thus add to $\blambda^{(\pos(\cT|i))}$ the $d$-node $(x+1,y+1,\pos({\mathcal{T}}|i))$ with the entry $i$ in it.
As the $d$-tableau $\mathcal{T}'$ is standard and $i_1,\,i_2 < i$, the result is a standard $d$-tableau of size $i$.

We have thus shown that $\cont \subseteq {\rm Im}f$, whence we deduce that ${\rm Im}f=\cont$. Therefore, the map $f$ induces a bijection between  
$\mathrm{STab}_d(n)$ and $\cont$.
\end{proof}

Here is an example of the correspondence of  Proposition \ref{cont-tab}. The following $2\times 10$ array of ${\mathrm{Cont}}_3(10)$
\[\left(\begin{array}{ccccccccccccccccccc}\xi_1&,&\xi_1&,&\xi_3&,&\xi_1&,&\xi_3&,&\xi_1&,&\xi_1&,&
\xi_2&,&\xi_1&,&\xi_3\\[.3em]
1&,&q^2&,&1&,&q^4&,&q^{-2}&,&q^{-2}&,&q^{-4}&,&1&,&1&,&q^2
         \end{array}\right),\]
where $\{ \xi_1,\xi_2,\xi_3\}$ is the set of all cubic roots of unity, corresponds to the following standard $3$-tableau of size $10$:
$$\left(\begin{array}{l} \fbox{\scriptsize{1}}\fbox{\scriptsize{2}}\fbox{\scriptsize{4}}\\[-0.14em]
\fbox{\scriptsize{6}}\fbox{\scriptsize{9}}\\[-0.18em]
\fbox{\scriptsize{7}}\end{array}\ ,\ \  \fbox{\scriptsize{8}}\ \ ,
\ \begin{array}{l} \fbox{\scriptsize{3}}\fbox{\scriptsize{10}}\\[-0.18em]
\fbox{\scriptsize{5}}\end{array}\right)\ .$$

\section{Representations of the Yokonuma--Hecke algebra}\label{sec-rep}

In this section, we 
will 
construct explicitly the irreducible representations of the Yokonuma--Hecke algebra $\C(q){\rm Y}_{d,n}(q)$. The study of the representations of $\C(q)\widehat{\rm Y}_{d,2}(q)$  in \S \ref{affine_{d,2}} suggests the action, for $i=1,\dots,n-1$, of the generators $g_i$, $t_i$ and $t_{i+1}$ of $\C(q){\rm Y}_{d,n}(q)$ on vectors indexed by the standard $d$-tableaux of size $n$. 
We will verify 
that this action extends to an action of the whole algebra $\C(q){\rm Y}_{d,n}(q)$, 
and then prove that all representations of $\C(q){\rm Y}_{d,n}(q)$ are obtained this way. 

\subsection{Formulas for the representations of $\C(q){\rm Y}_{d,n}(q)$}
For any $d$-tableau $\cT$ of size $n$ and any permutation $\sigma \in \mathfrak{S}_n$, we denote by $\cT^{\sigma}$ the $d$-tableau obtained from $\cT$ by applying the permutation $\sigma$ on the numbers contained in the $d$-nodes of $\cT$. We have 
\begin{equation}\label{pos-con}\pos(\cT^{\sigma}|i)=\pos\bigl(\cT|\sigma^{-1}(i)\bigr)\ \ \ \text{and}\ \ \ \cc(\cT^{\sigma}|i)=\cc\bigl(\cT|\sigma^{-1}(i)\bigr)\ \ \ \ \ \ \text{for all $i=1,\dots,n$.}\end{equation}
Note that if the $d$-tableau $\cT$ is standard, the $d$-tableau $\cT^{\sigma}$ is not necessarily standard. 

Let $\{\xi_1,\dots,\xi_d\}$ be the set of all $d$-th roots of unity (ordered arbitrarily). Let  $\mathcal{P}(d,n)$ be the set of all $d$-partitions of $n$, and let
 $\blambda \in \mathcal{P}(d,n)$. Let $V_{\blambda}$ be a  $\C(q)$-vector space with a basis $\{\bv_{_{\cT}}\}$ indexed by the standard $d$-tableaux of shape $\blambda$. We set $\bv_{_{\cT}}:=0$ for any non-standard $d$-tableau $\cT$ of shape $\blambda$.

\begin{prop}\label{prop-rep} 
Let $\cT$ be a standard $d$-tableau of shape $\blambda  \in \mathcal{P}(d,n)$. For brevity, we set $\pos_i:=\pos(\cT|i)$ and $\cc_i:=\cc(\cT|i)$ for $i=1,\dots,n$. 
The vector space $V_{\blambda}$ is a representation of $\C(q){\rm Y}_{d,n}(q)$ with the action of the generators on the basis element $\bv_{_{\cT}}$ defined as follows:
for $j=1,\dots,n$,
\begin{equation}\label{rep-t}
t_j(\bv_{_{\cT}})=\xi_{\pos_j}\bv_{_{\cT}}\  ;
\end{equation}
for $i=1,\dots,n-1$, if $\pos_{i} \neq \pos_{i+1}$ then
\begin{equation}\label{rep-g1}
g_i(\bv_{_{\cT}})=\bv_{_{\cT^{s_i}}}\ ,
\end{equation}
and if $\pos_{i}=\pos_{i+1}$ then
\begin{equation}\label{rep-g2}
g_i(\bv_{_{\cT}})=\frac{\cc_{i+1}(q-q^{-1})}{\cc_{i+1}-\cc_i}\,\bv_{_{\cT}}+\frac{q\cc_{i+1}-q^{-1}\cc_i}{\cc_{i+1}-\cc_i}\,\bv_{_{\cT^{s_i}}}\ ,
\end{equation}
where $s_i$ is the transposition $(i,i+1)$.
\end{prop}
\begin{proof}
We verify that Formulas (\ref{rep-t})--(\ref{rep-g2}) define a representation of the algebra $\C(q){\rm Y}_{d,n}(q)$ by proving that the given actions
satisfy the defining relations (\ref{modular}) and (\ref{quadr}) of ${\rm Y}_{d,n}(q)$. 
During the calculations, we will use several times Formulas (\ref{pos-con}) without mentioning.

Note first that, as the $d$-tableau $\cT$ is standard, then the $d$-tableau $\cT^{s_i}$ is not standard if and only if $\pos_{i+1}=\pos_i$ and $\cc_{i+1}=\cc_iq^{\pm2}$. If $\pos_{i+1}=\pos_i$ and $\cc_{i+1}=\cc_iq^{2\epsilon}$, with $\epsilon=\pm1$, then Formula (\ref{rep-g2}) becomes $g_i(\bv_{_{\cT}})=\epsilon q^{\epsilon}\,\bv_{_{\cT}}\,$.

The study of representations of the affine Yokonuma--Hecke algebra $\C(q)\widehat{{\rm Y}}_{d,2}(q)$ in \S \ref{affine_{d,2}}  implies directly that the relations $g_i^2 = 1 + (q-q^{-1}) \, e_{i} \, g_i$, $g_it_i=t_{i+1}g_i$ and $g_it_{i+1}=t_ig_{i}$ are satisfied.

The relations $t_j^d=1$ and $t_it_j=t_jt_i$, $i,j=1,\dots,n$, are immediately verified. Moreover, as $g_i(\bv_{_{\cT}})$ is a linear combination of $\bv_{_{\cT}}$ and $\bv_{_{\cT^{s_i}}}$, the relations $g_it_j=t_jg_i$, for $i=1,\dots,n-1$ and $j=1,\dots,n$ such that 
$j\neq i,i+1$, are also verified.

So it remains to check the relations 
\begin{equation}\label{br1}
\text{$g_ig_j(\bv_{_{\cT}})=g_jg_i(\bv_{_{\cT}})$ \,for $i,j=1,\dots,n-1$ such that $|i-j|>1$}, 
\end{equation}
and the relations 
\begin{equation}\label{br2}
\text{$g_ig_{i+1}g_i(\bv_{_{\cT}})=g_{i+1}g_ig_{i+1}(\bv_{_{\cT}})$ \,for $i=1,\dots,n-2$.}
\end{equation}
We will first verify (\ref{br1}).

Note now that Formula (\ref{rep-g2}) does not depend on $d$ (since $\pos_i=\pos_{i+1}$) 
and coincides with the usual formula, established in \cite{Hoe}, for the action of the generators of the Iwahori--Hecke algebra of type $A$ on vectors labelled by usual standard tableaux.
Thus, for $i,j=1,\dots,n-1$ such that $|i-j|>1$, if $\pos_{i}=\pos_{i+1}$ and $\pos_{j}=\pos_{j+1}$ then it is well-known that $g_ig_j(\bv_{_{\cT}})=g_jg_i(\bv_{_{\cT}})$ (it can be shown by a straightforward calculation that we do not repeat here).

Assume that $\pos_i\neq\pos_{i+1}$ and $\pos_j\neq\pos_{j+1}$. Then the $d$-tableaux $\cT^{s_i}$, $\cT^{s_j}$, $\cT^{s_is_j}$ and $\cT^{s_js_i}$ are standard, and we have
$\cT^{s_is_j}=\cT^{s_js_i}$ , since $s_is_j=s_js_i$. We deduce that
$g_ig_j(\bv_{_{\cT}})=\bv_{_{\cT^{s_js_i}}}=\bv_{_{\cT^{s_is_j}}}=g_jg_i(\bv_{_{\cT}})$.

Now assume that $\pos_i\neq\pos_{i+1}$ and $\pos_j=\pos_{j+1}$. Then the $d$-tableau $\cT^{s_i}$ is standard and a direct calculation shows that
\[g_ig_j(\bv_{_{\cT}})
=\frac{\cc_{j+1}(q-q^{-1})}{\cc_{j+1}-\cc_j}\,\bv_{_{\cT^{s_i}}}+\frac{q\cc_{j+1}-q^{-1}\cc_j}{\cc_{j+1}-\cc_j}\,g_i(\bv_{_{\cT^{s_j}}})\ \]
and
\[g_jg_i(\bv_{_{\cT}})=\frac{\cc_{j+1}(q-q^{-1})}{\cc_{j+1}-\cc_j}\,\bv_{_{\cT^{s_i}}}+\frac{q\cc_{j+1}-q^{-1}\cc_j}{\cc_{j+1}-\cc_j}\,\bv_{_{\cT^{s_is_j}}}\ .\]
 If the $d$-tableau $\cT^{s_j}$ is standard, then $g_i(\bv_{_{\cT^{s_j}}})=\bv_{_{\cT^{s_js_i}}}$. If $\cT^{s_j}$ is  not standard, then $\cT^{s_js_i}$ is not standard either, and 
 $g_i(\bv_{_{\cT^{s_j}}})=0=\bv_{_{\cT^{s_js_i}}}$.
 In any case,
 since $s_is_j=s_js_i$, we have $g_ig_j(\bv_{_{\cT}})=g_jg_i(\bv_{_{\cT}})$.
We deal similarly with 
the case $\pos_i=\pos_{i+1}$ and $\pos_j\neq\pos_{j+1}$.

We will now prove (\ref{br2}).
Let $i=1,\dots,n-2$. If $\pos_i=\pos_{i+1}=\pos_{i+2}$, then, again from the representation theory of the Iwahori--Hecke algebra of type $A$, it is well-known that $g_ig_{i+1}g_i(\bv_{_{\cT}})=g_{i+1}g_ig_{i+1}(\bv_{_{\cT}})$ (thus, we do not  have to repeat the straightforward, but lengthy, calculation).

Assume that $\pos_i$, $\pos_{i+1}$ and $\pos_{i+2}$ are all different. Then the $d$-tableaux $\cT^{s_i}$, $\cT^{s_is_{i+1}}$, $\cT^{s_is_{i+1}s_i}$, $\cT^{s_{i+1}}$, $\cT^{s_{i+1}s_i}$ and $\cT^{s_{i+1}s_is_{i+1}}$ are all standard, and we have $\cT^{s_is_{i+1}s_i}=\cT^{s_{i+1}s_is_{i+1}}$, since $s_is_{i+1}s_i=s_{i+1}s_is_{i+1}$.
We deduce that
 $g_ig_{i+1}g_i(\bv_{_{\cT}})=\bv_{_{\cT^{s_is_{i+1}s_i}}}=\bv_{_{\cT^{s_{i+1}s_is_{i+1}}}}=g_{i+1}g_ig_{i+1}(\bv_{_{\cT}})$.

Assume that $\pos_i=\pos_{i+1}\neq\pos_{i+2}$. Then the $d$-tableaux $\cT^{s_{i+1}}$ and $\cT^{s_{i+1}s_i}$ are 
standard. We calculate
\[\begin{array}{ll}
g_ig_{i+1}g_i(\bv_{_{\cT}}) & =\  \displaystyle g_ig_{i+1}\Bigl( \frac{\cc_{i+1}(q-q^{-1})}{\cc_{i+1}-\cc_i}\,\bv_{_{\cT}}+\frac{q\cc_{i+1}-q^{-1}\cc_i}{\cc_{i+1}-\cc_i}\,\bv_{_{\cT^{s_i}}}\Bigr)\\[1em]
 & =\ \displaystyle g_i \Bigl( \frac{\cc_{i+1}(q-q^{-1})}{\cc_{i+1}-\cc_i}\,\bv_{_{\cT^{s_{i+1}}}}+\frac{q\cc_{i+1}-q^{-1}\cc_i}{\cc_{i+1}-\cc_i}\,g_{i+1}(\bv_{_{\cT^{s_i}}})\Bigr)\\[1em]
 & =\  \displaystyle\frac{\cc_{i+1}(q-q^{-1})}{\cc_{i+1}-\cc_i}\,\bv_{_{\cT^{s_{i+1}s_i}}}+\frac{q\cc_{i+1}-q^{-1}\cc_i}{\cc_{i+1}-\cc_i}\,g_ig_{i+1}(\bv_{_{\cT^{s_i}}})\ ;\\[0.8em]
\end{array}\]
If $\cT^{s_i}$ is standard, then $\cT^{s_is_{i+1}}$ and $\cT^{s_is_{i+1}s_i}$ are standard, and $g_ig_{i+1}(\bv_{_{\cT^{s_i}}})=\bv_{_{\cT^{s_is_{i+1}s_i}}}$.
If $\cT^{s_i}$ is not standard, then $\cT^{s_is_{i+1}}$ and $\cT^{s_is_{i+1}s_i}$ are not standard either, and $g_ig_{i+1}(\bv_{_{\cT^{s_i}}})=0=\bv_{_{\cT^{s_is_{i+1}s_i}}}$.
In any case, we have
\[\begin{array}{ll}
g_ig_{i+1}g_i(\bv_{_{\cT}}) & =\  \displaystyle\frac{\cc_{i+1}(q-q^{-1})}{\cc_{i+1}-\cc_i}\,\bv_{_{\cT^{s_{i+1}s_i}}}+\frac{q\cc_{i+1}-q^{-1}\cc_i}{\cc_{i+1}-\cc_i}\,\bv_{_{\cT^{s_is_{i+1}s_i}}}\ .\\[0.8em]
\end{array}\]
On the other hand, we directly calculate
\[
g_{i+1}g_{i}g_{i+1}(\bv_{_{\cT}})  
=\  \displaystyle\frac{\cc_{i+1}(q-q^{-1})}{\cc_{i+1}-\cc_i}\,\bv_{_{\cT^{s_{i+1}s_i}}}+\frac{q\cc_{i+1}-q^{-1}\cc_i}{\cc_{i+1}-\cc_i}\,\bv_{_{\cT^{s_{i+1}s_is_{i+1}}}}\ .
\]
As $s_is_{i+1}s_i=s_{i+1}s_is_{i+1}$, we obtain that $g_ig_{i+1}g_i(\bv_{_{\cT}})=g_{i+1}g_ig_{i+1}(\bv_{_{\cT}})$.

The case $\pos_i\neq \pos_{i+1}=\pos_{i+2}$ is similar to the above, so we skip the details.

Finally, assume that $\pos_i=\pos_{i+2}\neq\pos_{i+1}$. Then the $d$-tableaux $\cT^{s_i}$ and $\cT^{s_{i+1}}$ are 
standard. We calculate
\[\begin{array}{ll}
g_ig_{i+1}g_i(\bv_{_{\cT}}) & =\  \displaystyle g_ig_{i+1}\bigl(\bv_{_{\cT^{s_i}}}\bigr)\\[0.6em]
 & =\ \displaystyle g_i \Bigl( \frac{\cc_{i+2}(q-q^{-1})}{\cc_{i+2}-\cc_i}\,\bv_{_{\cT^{s_{i}}}}+\frac{q\cc_{i+2}-q^{-1}\cc_i}{\cc_{i+2}-\cc_i}\,\bv_{_{\cT^{s_is_{i+1}}}}\Bigr)\\[1.2em]
 & =\  \displaystyle\frac{\cc_{i+2}(q-q^{-1})}{\cc_{i+2}-\cc_i}\,\bv_{_{\cT}}+\frac{q\cc_{i+2}-q^{-1}\cc_i}{\cc_{i+2}-\cc_i}\,g_i(\bv_{_{\cT^{s_is_{i+1}}}})
\end{array}\]
and
\[\begin{array}{ll}
g_{i+1}g_{i}g_{i+1}(\bv_{_{\cT}}) & = \ \displaystyle g_{i+1}g_{i}\bigl(\bv_{_{\cT^{s_{i+1}}}}\bigr)\\[0.6em]
 & =\ \displaystyle g_{i+1} \Bigl( \frac{\cc_{i+2}(q-q^{-1})}{\cc_{i+2}-\cc_i}\,\bv_{_{\cT^{s_{i+1}}}}+\frac{q\cc_{i+2}-q^{-1}\cc_i}{\cc_{i+2}-\cc_i}\,\bv_{_{\cT^{s_{i+1}s_{i}}}}\Bigr)\\[1.2em]
 & 
 =\  \displaystyle\frac{\cc_{i+2}(q-q^{-1})}{\cc_{i+2}-\cc_i}\,\bv_{_{\cT}}+\frac{q\cc_{i+2}-q^{-1}\cc_i}{\cc_{i+2}-\cc_i}\,g_{i+1}(\bv_{_{\cT^{s_{i+1}s_i}}})\ .\\[0.8em]
\end{array}
\]
If $\cT^{s_is_{i+1}}$ is standard, then $g_i(\bv_{_{\cT^{s_is_{i+1}}}})=\bv_{_{\cT^{s_is_{i+1}s_i}}}$.
If  $\cT^{s_is_{i+1}}$ is not standard, then $\cT^{s_is_{i+1}s_i}$ is not standard either, and $g_i(\bv_{_{\cT^{s_is_{i+1}}}})=0=\bv_{_{\cT^{s_is_{i+1}s_i}}}$.
Similarly, if $\cT^{s_{i+1}s_{i}}$ is standard, then $g_{i+1}(\bv_{_{\cT^{s_{i+1}s_i}}})=\bv_{_{\cT^{s_{i+1}s_is_{i+1}}}}$.
If $\cT^{s_{i+1}s_{i}}$ is not standard, then $\cT^{s_{i+1}s_is_{i+1}}$ is not standard either, and $g_{i+1}(\bv_{_{\cT^{s_{i+1}s_i}}})=0=\bv_{_{\cT^{s_{i+1}s_is_{i+1}}}}$.
 As $s_is_{i+1}s_i=s_{i+1}s_is_{i+1}$, we obtain that $g_ig_{i+1}g_i(\bv_{_{\cT}})=g_{i+1}g_ig_{i+1}(\bv_{_{\cT}})$.
\end{proof}

\begin{prop}\label{J-rep}
Let $\blambda$ be a $d$-partition of $n$. The action of the Jucys--Murphy elements $J_1,\dots,J_n$ of the algebra $\C(q){\rm Y}_{d,n}(q)$ on the vector space $V_{\blambda}$ is given by
\begin{equation}\label{action-JM}
J_i(\bv_{_{\cT}})=\cc(\cT|i) \,\bv_{_{\cT}}\ \quad\text{{ for any standard $d$-tableau $\cT$ of shape $\blambda$.}}
\end{equation}
\end{prop}
\begin{proof} Let $\cT$ be a standard $d$-tableau of shape $\blambda$, and set $\pos_i:=\pos(\cT|i)$ and $\cc_i:=\cc(\cT|i)$ for $i=1,\dots,n$. We will prove (\ref{action-JM}) by induction on $i$. Recall that $J_{i+1}=g_iJ_ig_i$ for all $i=1,\dots,n-1$.

For $i=1$, (\ref{action-JM}) is satisfied, since $J_1=1$ and $\cc_1=1$.
Let $i>1$, and assume that (\ref{action-JM}) holds for smaller values of $i$.

First suppose that
 $\pos_{i}\neq\pos_{i+1}$. Then we have, from (\ref{rep-g1}), that
\[J_{i+1}(\bv_{_{\cT}})=g_iJ_ig_i(\bv_{_{\cT}})=g_iJ_i(\bv_{_{\cT^{s_i}}})= g_i(\cc_{i+1}\,\bv_{_{\cT^{s_i}}})=\cc_{i+1}\,g_i(\bv_{_{\cT^{s_i}}})=\cc_{i+1}\,\bv_{_{\cT}}\ ,\]
where we have used the induction hypothesis together with the fact that $\cc(\cT^{s_i}|i)=\cc_{i+1}$.

Now suppose that $\pos_{i}=\pos_{i+1}$. Then we have, from (\ref{rep-g2}), that
\[\begin{array}{ll}J_{i+1}(\bv_{_{\cT}}) & =g_iJ_ig_i(\bv_{_{\cT}})\\[0.4em]
 & =\displaystyle g_iJ_i\Bigl(\frac{\cc_{i+1}(q-q^{-1})}{\cc_{i+1}-\cc_i}\,\bv_{_{\cT}}+\frac{q\cc_{i+1}-q^{-1}\cc_i}{\cc_{i+1}-\cc_i}\,\bv_{_{\cT^{s_i}}}\Bigr)\\[1.2em]
 & =\displaystyle g_i\Bigl(\frac{\cc_i\cc_{i+1}(q-q^{-1})}{\cc_{i+1}-\cc_i}\,\bv_{_{\cT}}+\frac{\cc_{i+1}(q\cc_{i+1}-q^{-1}\cc_i)}{\cc_{i+1}-\cc_i}\,\bv_{_{\cT^{s_i}}}\Bigr)\\[1.4em]
& =\displaystyle \frac{\cc_i\cc_{i+1}(q-q^{-1})}{\cc_{i+1}-\cc_i}\Bigl(\frac{\cc_{i+1}(q-q^{-1})}{\cc_{i+1}-\cc_i}\,\bv_{_{\cT}}+\frac{q\cc_{i+1}-q^{-1}\cc_i}{\cc_{i+1}-\cc_i}\,\bv_{_{\cT^{s_i}}}\Bigr)\\[1em]
 & \ \ +\ \  \displaystyle \frac{\cc_{i+1}(q\cc_{i+1}-q^{-1}\cc_i)}{\cc_{i+1}-\cc_i}\Bigl(\frac{\cc_{i}(q-q^{-1})}{\cc_{i}-\cc_{i+1}}\,\bv_{_{\cT^{s_i}}}+\frac{q\cc_{i}-q^{-1}\cc_{i+1}}{\cc_{i}-\cc_{i+1}}\,\bv_{_{\cT}}\Bigr)\ ,\\
 \end{array}
\]
where we have used the induction hypothesis together with the fact that $\cc(\cT^{s_i}|i)=\cc_{i+1}$ and $\cc(\cT^{s_i}|i+1)=\cc_{i}$. First, we observe that the coefficient in front of $\bv_{_{\cT^{s_i}}}$ is equal to $0$. Second, we note that if $\cc_{i+1}=\cc_iq^{\pm2}$, then the last term of the sum is equal to 0. Thus, if the $d$-tableau $\cT^{s_i}$ is not standard (that is, if $\cc_{i+1}=\cc_iq^{\pm2}$), the calculation above is still valid. A direct calculation shows that $J_{i+1}(\bv_{_{\cT}})=\cc_{i+1}\,\bv_{_{\cT}}$.
\end{proof}

Let $\cT \in \mathrm{STab}_d(n)$ and let $\Lambda$ be the corresponding content array, through the identification of $\mathrm{STab}_d(n)$ with $\mathrm{Cont}_d(n)$ given by Proposition \ref{cont-tab}. Following the action of the Jucys--Murphy elements $t_1,\ldots,t_n,J_1,\dots,J_n$ on $\bv_{_{\cT}}$, given by Propositions \ref{prop-rep} and \ref{J-rep}, we obtain that 
$\Lambda \in \spec$, with corresponding vector $\bv_{_{\cT}}$.
Thus, we obtain an inclusion of the set $\mathrm{STab}_d(n)$  into the spectrum $\spec$ of the Jucys--Murphy elements
$t_1,\ldots,t_n,J_1,\dots,J_n$. On the other hand, Propositions \ref{spec-cont} and \ref{cont-tab} provide an inclusion of the set $\spec$ into $\mathrm{STab}_d(n)$. These operations, by construction, are inverse to each other. We sum up the results.

\begin{prop} \label{spec-tab}
The set $\mathrm{STab}_d(n)$ is in bijection with the set $\spec$.
\end{prop}

It remains to show that we have now constructed all irreducible representations for $\C(q){\rm Y}_{d,n}(q)$.

\begin{thm}\label{comp}
For any $\blambda  \in \mathcal{P}(d,n)$, let $V_{\blambda}$ denote the representation of $\C(q){\rm Y}_{d,n}(q)$ constructed in Proposition $\ref{prop-rep}$. Then
\begin{enumerate}[{\rm (a)}]
\item If $V_{\blambda}$ is isomorphic to $V_{\bmu}$ for some $\bmu  \in \mathcal{P}(d,n)$, then $\blambda=\bmu$. \smallbreak
\item The representation $V_{\blambda}$   is irreducible. \smallbreak
\item The set $\{V_{\blambda}\,|\, \blambda  \in \mathcal{P}(d,n)\}$ is a complete set of pairwise non-isomorphic irreducible representations of $\C(q){\rm Y}_{d,n}(q)$.
\end{enumerate}
\end{thm}
\begin{proof}
Part (a) derives from Proposition \ref{spec-tab}, given the action of the Jucys--Murphy elements on   $V_{\blambda}$.
For (b), we will proceed by induction on $n$.

If $n=1$, then $\C(q){\rm Y}_{d,1}(q)$ is isomorphic to the group algebra over $\C(q)$ of the cyclic group of order $d$, 
and its irreducible representations are the ones described by Proposition $\ref{prop-rep}$. 
Now assume that (b)  
holds for $\C(q){\rm Y}_{d,n-1}(q)$. 
Let $\blambda$ be a $d$-partition of $n$. We will show that  $V_{\blambda}$   is an irreducible representation of $\C(q){\rm Y}_{d,n}(q)$.

Recall that we denote by ${\mathcal{E}}_-(\blambda)$ the set of all removable $d$-nodes from $\blambda$.
Following 
the description of the action of the generators $t_1,\ldots,t_{n-1},\,g_1,\ldots,g_{n-2}$ on $V_{\blambda}$, we must have
\begin{equation}\label{BR}{\rm Res}_{\C(q){\rm Y}_{d,n-1}(q)}^{\C(q){\rm Y}_{d,n}(q)} \,(V_{\blambda}) = \bigoplus_{\btheta \in {\mathcal{E}}_-(\blambda)}  V_{\blambda \setminus \{\btheta\}}.\end{equation}
Suppose now  that there exists a proper submodule $M$ of $V_{\blambda}$ such that $M \neq\{0\}$. 
Then,
by induction hypothesis, there exists a non-empty subset $\mathcal{E}_-(M)$ of ${\mathcal{E}}_-(\blambda)$ such that
$${\rm Res}_{\C(q){\rm Y}_{d,n-1}(q)}^{\C(q){\rm Y}_{d,n}(q)} \,(M) = \bigoplus_{\btheta \in {\mathcal{E}}_-(M)}  V_{\blambda \setminus \{\btheta\}}.$$
Thus, if $\cT$ is a standard $d$-tableau of shape $\blambda$ such that the entry $n$ is contained in a $d$-node $\btheta \in   \mathcal{E}_-(M)$, then
${\bf v}_{_{\cT}} \in M$.  Moreover, since $M$ is a proper submodule of $V_{\blambda}$, there exists at least one $d$-node $\btheta'\in{\mathcal{E}}_-(\blambda)$ such that $\btheta'\notin{\mathcal{E}}_-(M)$. Let $\cT'$ be a standard $d$-tableau of shape $\blambda$ such that the number $n$ is in the $d$-node $\btheta'$. Then ${\bf v}_{_{\cT'}} \notin M$. 
Let $\sigma \in \mathfrak{S}_n$ be the permutation such that $\cT^\sigma=\cT'$. If $\sigma=s_{i_1}s_{i_2}\ldots s_{i_r}$, where $s_i$ is the transposition $(i,i+1)$, then there exists $k \in \{0,1,\ldots,r-1\}$ such that 
${\bf v}_{\cT^{s_{i_1}s_{i_2}\ldots s_{i_k}}} \in M$ and ${\bf v}_{\cT^{s_{i_1}s_{i_2}\ldots s_{i_{k+1}}}} \notin M$. We can thus replace $\cT, \cT'$ by
$\cT^{s_{i_1}s_{i_2}\ldots s_{i_k}},\cT^{s_{i_1}s_{i_2}\ldots s_{i_{k+1}}}$, and, by setting $i:=i_{k+1}$, we have that ${\bf v}_{_{\cT}} \in M$ and $\bv_{_{\cT^{s_i}}} \notin M$. Following the definition of the action of $\C(q){\rm Y}_{d,n}(q)$ on $M$ given by Proposition \ref{prop-rep}, this is only possible if
$$\pos(\cT|i)=\pos(\cT|i+1)\,\,\,\,\,\text{and}\,\,\,\,\,\cc(\cT|i)=q^2\,\cc(\cT|i+1).$$
However, in this case, as we have also already seen in the proof of Proposition \ref{prop-rep}, $\cT^{s_i}$ is not a standard $d$-tableau,
and so $\bv_{_{\cT^{s_i}}} = 0 \in M$,
a contradiction. We conclude that $V_{\blambda}$   is irreducible.

Finally, part (c)  follows by counting dimensions, since, from the Robinson-Schensted correspondence, we have 
\begin{equation}\label{Robinson-Schensted}
 \sum_{\blambda  \in \mathcal{P}(d,n)} \left({\rm dim}_{\C(q)} (V_{\blambda})\right)^2 = d^n n! = {\rm dim}_{\C(q)} (\C(q){\rm Y}_{d,n}(q)).
 \end{equation}\end{proof}

\begin{rem}\label{rem-BR} {\rm
Formula (\ref{BR}) gives the branching rules for the irreducible representations of the algebra $\C(q){\rm Y}_{d,n}(q)$ with respect to its subalgebra $\C(q){\rm Y}_{d,n-1}(q)$. These rules are simply expressed in the combinatorial terms of $d$-partitions, and can be equivalently stated by saying that the Bratteli diagram of the chain of  algebras $\C(q){\rm Y}_{d,n}(q)$ coincides with the Hasse diagram of the poset of $d$-partitions (ordered by inclusion). In particular, we have shown that the branching rules are ``multiplicity-free''; 
as it is well-known, this implies that the centraliser of $\C(q){\rm Y}_{d,n-1}(q)$ in $\C(q){\rm Y}_{d,n}(q)$ is commutative.
}\end{rem}

\subsection{Primitive idempotents and maximal commutative subalgebra of $\C(q){\rm Y}_{d,n}(q)$}\label{ET}

Following the construction of its irreducible representations, the algebra $\C(q){\rm Y}_{d,n}(q)$ is split. Moreover, Equation (\ref{Robinson-Schensted}) implies that the algebra $\C(q){\rm Y}_{d,n}(q)$ is also semisimple.

\begin{rem}{\rm
One may observe that ${\rm Y}_{d,n}(q)$ is split over a much smaller field, $\mathbb{Q}(\zeta,q)$, where $\zeta:={\rm exp}(2\pi i/d)$. Then all  constructions in this paper can be repeated by replacing $\C(q)$ with $\mathbb{Q}(\zeta,q)$.
}\end{rem}

For any $d$-partition $\blambda$ of $n$, denote by $m_{\blambda}$ the dimension of the representation $V_{\blambda}$. 
We fix the basis $\{\bv_{_{\cT}}\}$ of $V_{\blambda}$ used in Proposition \ref{prop-rep} and use it to identify $\text{End}_{\C(q)}(V_{\blambda})$ with the matrix algebra ${\rm Mat}_{m_{\blambda}}(\C(q))$ over  $\C(q)$. Since $\C(q){\rm Y}_{d,n}(q)$ is split semisimple, it follows from Theorem \ref{comp}(c) and the Artin--Wedderburn theorem that there exists an isomorphism
\begin{equation}\label{WA}I: \C(q){\rm Y}_{d,n}(q) \rightarrow \prod_{\blambda \in   \mathcal{P}(d,n)} {\rm Mat}_{m_{\blambda}}(\C(q)).\end{equation}
We write $I_{\blambda}$ for the projection of $I$ onto the $\blambda$-factor, that is,
$$I_{\blambda}: \C(q){\rm Y}_{d,n}(q) \twoheadrightarrow {\rm Mat}_{m_{\blambda}}(\C(q)).$$
Let $\cT$ be a standard $d$-tableau of shape $\blambda$. Since $I$ is an isomorphism, there exists a unique element
 $E_{_{\cT}}$  of $\C(q){\rm Y}_{d,n}(q)$ that satisfies:
$$I_{\bmu}(E_{_{\cT}})=
\left\{ \begin{array}{ll}
0 & \text{ if } \blambda \neq \bmu \,;\\
P_{\bv_{_{\cT}}}& \text{ if } \blambda = \bmu\,,
\end{array}\right.$$
where 
$P_{\bv_{_{\cT}}}$ stands for the projection onto $\C(q)\bv_{_{\cT}}$, that is, $P_{\bv_{_{\cT}}}$ is 
the diagonal $m_{\blambda} \times m_{\blambda}$ matrix with coefficient $1$ in the column labelled by $\bv_{_{\cT}}$, and $0$ everywhere else.
The set $\{P_{\bv_{_{\cT}}}\}$, where $\bv_{_{\cT}}$ runs over the basis vectors of $V_{\blambda}$, is a complete set of pairwise orthogonal primitive idempotents of ${\rm Mat}_{m_{\blambda}}(\C(q))$. 
Thus, the element  $E_{_{\cT}}$ is a primitive idempotent of $\C(q){\rm Y}_{d,n}(q)$, and the set $\{ E_{_{\cT}}\}_{\cT \in  \mathrm{STab}_d(n)}$ is a complete set of pairwise orthogonal primitive idempotents of $\C(q){\rm Y}_{d,n}(q)$.

Now, the Jucys--Murphy elements $t_1,\dots,t_n,J_1,\dots,J_n$ are represented by diagonal matrices in the basis $\{\bv_{_{\cT}}\}$ of $V_{\blambda}$ indexed by the standard $d$-tableaux of shape $\blambda$ (see Formulas (\ref{rep-t}) and (\ref{action-JM})). Moreover, the eigenvalues of the set $\{t_1,\dots,t_n,J_1,\dots,J_n\}$ allow to distinguish between all basis vectors $\bv_{_{\cT}}$ of all representations $V_{\blambda}$, with $\blambda \in \mathcal{P}(d,n)$ (Proposition \ref{spec-tab}). Thus, for $\cT$ a standard $d$-tableau of size $n$, we can express the primitive idempotent $E_{_{\cT}}$ of $\C(q){\rm Y}_{d,n}(q)$
in terms of the Jucys--Murphy elements $t_1,\dots,t_n$, $J_1,\dots,J_n$, as follows:

Let $\btheta$ be the $d$-node of $\cT$ with the number $n$ in it. As the tableau $\cT$ is standard, the $d$-node $\btheta$ is removable. Let ${\mathcal{U}}$ be the standard $d$-tableau obtained from $\cT$ by removing the $d$-node $\btheta$ and let $\bmu$ be the shape of ${\mathcal{U}}$.
The inductive formula for $E_{_{\cT}}$ in terms of the Jucys--Murphy elements reads:
\begin{equation}\label{idem-JM}
E_{_{\cT}}=E_{_{{\mathcal{U}}}}\prod_{ \begin{array}{l}\scriptstyle{\btheta'\in 
{\mathcal{E}}_+(\bmu)}\\\scriptstyle{\cc(\btheta')\neq \cc(\btheta)}\end{array}}\hspace{-0.2cm} \frac{J_n-\cc(\btheta')}{\cc(\btheta)-\cc(\btheta')}\prod_{ \begin{array}{l}\scriptstyle{\btheta'\in 
{\mathcal{E}}_+(\bmu)}\\\scriptstyle{\pos(\btheta')\neq \pos(\btheta)}\end{array}  }\hspace{-0.2cm}\frac{t_n-\xi_{\pos(\btheta')}}{\xi_{\pos(\btheta)}-\xi_{\pos(\btheta')}}\ ,
\end{equation}
with $E_{_{\cT_0}}=1$ for the unique $d$-tableau $\cT_0$ of size $0$. Note that, due to the commutativity of the Jucys--Murphy elements, all terms in the above formula commute with each other.

In Formula (\ref{idem-JM}), we consider the idempotent $E_{_{{\mathcal{U}}}}$ of $\C(q){\rm Y}_{d,n-1}(q)$ as an element of $\C(q){\rm Y}_{d,n}(q)$ thanks to the chain property (\ref{chain}) of the algebras ${\rm Y}_{d,n}(q)$. In fact,
seeing $E_{_{{\mathcal{U}}}}$ as an element of $\C(q){\rm Y}_{d,n}(q)$, we have
\begin{equation}\label{EU-ET}E_{_{{\mathcal{U}}}} = \sum_{ \begin{array}{c}\scriptstyle{\bpsi \in 
{\mathcal{E}}_+(\bmu)}\end{array}}\hspace{-0.2cm} 
E_{_{\mathcal{U} \cup \{\bpsi\}}}\,\,\,,\end{equation}
where, for any $\bpsi \in {\mathcal{E}}_+(\bmu)$, $\mathcal{U} \cup \{\bpsi\}$ is the standard $d$-tableau obtained from $\mathcal{U}$ by adding the $d$-node $\bpsi$ with the number $n$ in it. 
Then
$$E_{_{{\mathcal{U}}}}\prod_{ \begin{array}{l}\scriptstyle{\btheta'\in 
{\mathcal{E}}_+(\bmu)}\\\scriptstyle{\cc(\btheta')\neq \cc(\btheta)}\end{array}}\hspace{-0.2cm} \frac{J_n-\cc(\btheta')}{\cc(\btheta)-\cc(\btheta')} = 
\sum_{ \begin{array}{c}\scriptstyle{\bpsi\in 
{\mathcal{E}}_+(\bmu)}\\\scriptstyle{\cc(\bpsi)= \cc(\btheta)}\end{array}}\hspace{-0.2cm} 
E_{_{\mathcal{U} \cup \{\bpsi\}}}\,\,\,.$$
Similarly, 
$$E_{_{{\mathcal{U}}}}\prod_{ \begin{array}{l}\scriptstyle{\btheta'\in 
{\mathcal{E}}_+(\bmu)}\\\scriptstyle{\pos(\btheta')\neq \pos(\btheta)}\end{array}}\hspace{-0.2cm} \frac{t_n-\xi_{\pos(\btheta')}}{\xi_{\pos(\btheta)}-\xi_{\pos(\btheta')}} = 
\sum_{ \begin{array}{c}\scriptstyle{\bpsi\in 
{\mathcal{E}}_+(\bmu)}\\\scriptstyle{\pos(\bpsi)= \pos(\btheta)}\end{array}}\hspace{-0.2cm} 
E_{_{\mathcal{U} \cup \{\bpsi\}}}\,\,\,.$$
Thus,
$$E_{_{{\mathcal{U}}}}\prod_{ \begin{array}{l}\scriptstyle{\btheta'\in 
{\mathcal{E}}_+(\bmu)}\\\scriptstyle{\cc(\btheta')\neq \cc(\btheta)}\end{array}}\hspace{-0.2cm} \frac{J_n-\cc(\btheta')}{\cc(\btheta)-\cc(\btheta')}\prod_{ \begin{array}{l}\scriptstyle{\btheta'\in 
{\mathcal{E}}_+(\bmu)}\\\scriptstyle{\pos(\btheta')\neq \pos(\btheta)}\end{array}  }\hspace{-0.2cm}\frac{t_n-\xi_{\pos(\btheta')}}{\xi_{\pos(\btheta)}-\xi_{\pos(\btheta')}}\
=\sum_{ \begin{array}{c}\scriptstyle{\bpsi\in 
{\mathcal{E}}_+(\bmu)}\\\scriptstyle{\cc(\bpsi)= \cc(\btheta)}\\\scriptstyle{\pos(\bpsi)= \pos(\btheta)}\end{array}}\hspace{-0.2cm} 
E_{_{\mathcal{U} \cup \{\bpsi\}}}\,\,\,=\,\,\,E_{_{\cT}}.$$

Thanks  to Formula (\ref{idem-JM}), we can now prove the following result, which is just a corollary of Theorem \ref{comp}.

\begin{prop}\label{maxcommutative}
The Jucys--Murphy elements $t_1,\dots,t_n,J_1,\dots,J_n$ generate a maximal commutative subalgebra of $\C(q){\rm Y}_{d,n}(q)$.
\end{prop}
\begin{proof}
The Jucys--Murphy elements $t_1,\dots,t_n,J_1,\dots,J_n$ generate a commutative subalgebra of $\C(q){\rm Y}_{d,n}(q)$ (Corollary \ref{commutativeset}). Due to (\ref{idem-JM}), all the primitive idempotents $E_{_{\cT}}$, indexed by the standard $d$-tableaux of size $n$, belong to this subalgebra. We have already seen that these idempotents form a complete set of primitive idempotents of $\C(q){\rm Y}_{d,n}(q)$, that is, they span a maximal commutative subalgebra of $\C(q){\rm Y}_{d,n}(q)$. As a conclusion, the subalgebra generated by the Jucys--Murphy elements $t_1,\dots,t_n,J_1,\dots,J_n$ coincides with the maximal commutative subalgebra spanned by all idempotents $E_{_{\cT}}$.
\end{proof}

\begin{cor}\label{maxcommutativeoverring}
The Jucys--Murphy elements $t_1,\dots,t_n,J_1,\dots,J_n$ generate a maximal commutative subalgebra of ${\rm Y}_{d,n}(q)$.
\end{cor}

If now we set
\begin{equation}\label{block}E_{\blambda}: = \sum_{ \begin{array}{l}\scriptstyle{\cT \in \mathrm{STab}_d(n)}\\\scriptstyle{{\rm Shape(\cT) = \blambda}}\end{array}}\hspace{-0.2cm} 
E_{_{\cT}}\,\,\,,\end{equation}
then $E_{\blambda}$ is a primitive idempotent of the centre $Z(\C(q){\rm Y}_{d,n}(q))$ of $\C(q){\rm Y}_{d,n}(q)$.
The primitive idempotents of  $Z(\C(q){\rm Y}_{d,n}(q))$ are called {\em block-idempotents} or simply {\em blocks} of the Yokonuma--Hecke algebra $\C(q){\rm Y}_{d,n}(q)$.
We have that the set $\{E_{\blambda}\}_{\blambda \in  \mathcal{P}(d,n)}$ is the set of all block-idempotents of $\C(q){\rm Y}_{d,n}(q)$,
and thus a basis of the subalgebra $Z(\C(q){\rm Y}_{d,n}(q))$. 

Let $\mathfrak{A}_n$ be the subalgebra of $\C(q){\rm Y}_{d,n}(q)$ generated by the union of the centres $Z(\C(q){\rm Y}_{d,i}(q))$,  for $i=1,\dots,n$ (the chain property (\ref{chain}) allows us to consider $Z(\C(q){\rm Y}_{d,i}(q))$ as a subset of $\C(q){\rm Y}_{d,n}(q)$).
We have seen (Remark \ref{rem-BR}) that the branching rules for the chain of  semisimple algebras $\C(q){\rm Y}_{d,n}(q)$ are multiplicity-free. A well-known consequence is that $\mathfrak{A}_n$ is a maximal commutative subalgebra of $\C(q){\rm Y}_{d,n}(q)$ (see, for example, \cite{OV}). Moreover, $\mathfrak{A}_n$ coincides with the subalgebra of $\C(q){\rm Y}_{d,n}(q)$ generated by the union of the centralisers of $\C(q){\rm Y}_{d,i}(q)$ in $\C(q){\rm Y}_{d,i+1}(q)$, for $i=1,\dots,n-1$.

Now, Formulas (\ref{block}) and (\ref{EU-ET}) imply that $\mathfrak{A}_n$ is contained in the subalgebra of $\C(q){\rm Y}_{d,n}(q)$  spanned by the primitive idempotents $E_{_{\mathcal{T}}}$, where $\mathcal{T}$ runs through the set of standard $d$-tableaux of size $n$. The maximality of $\mathfrak{A}_n$ together with the proof of Proposition \ref{maxcommutative} yield the following result:

\begin{cor}The subalgebra of $\C(q){\rm Y}_{d,n}(q)$ generated by the Jucys--Murphy elements  $t_1,\dots,t_n$, $J_1,\dots,J_n$ coincides with:
\begin{enumerate}[(a)]
\item  the subalgebra generated by the union of the centres $Z(\C(q){\rm Y}_{d,i}(q))$, for $i=1,\dots,n$ ; \smallbreak
\item the subalgebra generated by the union of the centralisers of $\C(q){\rm Y}_{d,i}(q)$ in $\C(q){\rm Y}_{d,i+1}(q)$, for $i=1,\dots,n-1$.
\end{enumerate}
\end{cor}

\section{A semisimplicity criterion for the Yokonuma--Hecke algebra}

Now note that all constructions so far would have worked 
if we had taken $q$ to be a non-zero complex number that is not a root of unity, instead of taking $q$ to be an indeterminate. If $q$ were a root of unity, then the construction of the representation $ V_{\blambda}$ would fail in the following case
(see Formula (\ref{rep-g2})): there exists a standard $d$-tableau $\mathcal{T}$ of shape $\blambda$ and some $i \in \{1,\ldots,n-1\}$ such that
$$ \pos(\cT|i) = \pos(\cT|i+1) \,\,\,\,\,\text{and}\,\,\,\,\, \cc(\cT|i)=\cc(\cT|i+1).
\footnote{In the case where $q$ is an indeterminate or $q$ is not a root of unity,  
this condition implies that the numbers $i$ and $i+1$ are in the same diagonal of $\blambda^{(\pos(\cT|i))}$, which contradicts the fact that $\cT$ is a standard $d$-tableau.}$$
This in fact may happen only if $q^{2m}=1$ for some $m \in \mathbb{Z}$ such that $ 0 < m \leq \blambda^{(\pos(\cT|i))}_1+\blambda_1^{(\pos(\cT|i))'}-2$, where $'$ stands for the conjugate partition. Moreover,
the part of the proof of Theorem \ref{comp} showing that $V_{\blambda}$ is irreducible does not work 
in the following case:
there exists a standard $d$-tableau $\mathcal{T}$ of shape $\blambda$ and some $i \in \{1,\ldots,n-1\}$ such that $\cT^{s_i}$ is a standard $d$-tableau, and
$$ \pos(\cT|i) = \pos(\cT|i+1) \,\,\,\,\,\text{and}\,\,\,\,\, \cc(\cT|i)=q^2\,\cc(\cT|i+1).$$
This  may happen only if $q^{2m}=1$ for some $m \in \mathbb{Z}$ such that $ 0 < m \leq \blambda^{(\pos(\cT|i))}_1+\blambda_1^{(\pos(\cT|i))'}-1$.

We conclude that the algebra  $\C(q){\rm Y}_{d,n}(q)$ is split semisimple and its irreducible representations are the ones described by Proposition \ref{prop-rep} when $q$ is either an indeterminate or a non-zero complex number, unless the following holds: 
there exists a standard $d$-tableau $\mathcal{T}$ of shape $\blambda \in \mathcal{P}(d,n)$, 
$i \in \{1,2,\ldots,n-1\}$ and $m \in \{1,2,\ldots,\blambda^{(\pos(\cT|i))}_1+\blambda_1^{(\pos(\cT|i))'}-1\}$ such that $q^{2m}=1$. 
Note that $\blambda^{(\pos(\cT|i))}_1+\blambda_1^{(\pos(\cT|i))'}-1$ is equal to the  hook length of the $d$-node $(1,1,\pos(\cT|i))$ (cf.~Formula (\ref{hl})), and that
$${\rm max}  \left\{ \, \mathrm{hl}({\btheta}) \,\,|\,\, \btheta \text{ is a $d$-node of a $d$-partition of $n$ } \right\}=n.$$ 

We are now ready to prove a semisimplicity criterion for the Yokonuma--Hecke algebra  ${\rm Y}_{d,n}(q)$.

\begin{prop}\label{semisimplicity}
Let  $\vartheta: \mathbb{C}[q,q^{-1}] \rightarrow \mathbb{C}$ be a ring homomorphism such that $\vartheta(q)=\bar{q} \in \mathbb{C}\setminus \{0\}$.
We consider the specialised Yokonuma--Hecke algebra ${\rm Y}_{d,n}(\bar{q}):= \mathbb{C} \otimes_{ \mathbb{C}[q,q^{-1}]}  {\rm Y}_{d,n}(q)$, defined via $\vartheta$.
The algebra ${\rm Y}_{d,n}(\bar{q})$ is split semisimple if and only if $\vartheta\,(P(q))\neq 0$, where 
$$ P(q)=\prod_{m=1}^{n} (q^{2(m-1)}+q^{2(m-2)}+ \cdots+ q^2+1).$$
\end{prop}

\begin{proof}
First note that
$$ P(q)=\prod_{m=1}^{n} \frac{q^{2m}-1}{q^2-1}.$$
Following the discussion before Proposition \ref{semisimplicity}, if $\bar{q}^{2m}\neq1$ for all $m=1,2,\ldots, n$, then ${\rm Y}_{d,n}(\bar{q})$ is split semisimple. Moreover, if
$\bar{q}^2=1$, then  ${\rm Y}_{d,n}(\pm 1) \cong \C[G(d,1,n)]$ is split semisimple as well. Hence, if  $\vartheta\,(P(q))\neq 0$, then ${\rm Y}_{d,n}(\bar{q})$ is split semisimple.

Now assume that  $\vartheta\,(P(q))=0$. In this case, we know by \cite{GyUn} that the Iwahori--Hecke algebra of type $A$, ${\rm Y}_{1,n}(\bar{q})$, is not semisimple. Then we can take
an indecomposable, non-irreducible   ${\rm Y}_{1,n}(\bar{q})$-module $M$, and turn it into a  ${\rm Y}_{d,n}(\bar{q})$-module, by defining the action of $t_j$ on $M$ to be the identity for all $j=1,\ldots,n$.
Then $M$ becomes an indecomposable, non-irreducible   ${\rm Y}_{d,n}(\bar{q})$-module, and so  ${\rm Y}_{d,n}(\bar{q})$ is not semisimple.
\end{proof}

\section{Schur elements for the Yokonuma--Hecke algebra}

In this section we will define a symmetrising form for the Yokonuma--Hecke algebra and calculate the Schur elements with respect to that symmetrising form. 

\subsection{Preliminaries on symmetric algebras}
Let $R$ be a ring and let $A$ be an $R$-algebra, free and finitely generated as an $R$-module. A {\em symmetrising form} on $A$ is a linear map $\tau: A \rightarrow R$ such that
\begin{itemize}
\item $\tau(aa')=\tau(a'a)$ for all $a,a' \in A $, that is, $\tau$ is a {\em trace function};
\item the bilinear form $A  \times A \rightarrow  R,\, (a,a') \mapsto \tau(aa')$ is non-degenerate, that is, the determinant of the matrix $(\tau(bb'))_{b\in B,b'\in B'}$ is a unit in $R$ for some (and hence every) $R$-bases $B$ and $B'$ of $A$.
\end{itemize} 

If there exists a symmetrising form $\tau$ on $A$, then the algebra $A$ is called {\em symmetric}. We denote by $B^\vee = \{b^\vee\,|\,b \in B\}$ the {\em dual basis to $B$ with respect to $\tau$}; it is uniquely determined by the requirement that $\tau(b^\vee b')=\delta_{b,b'}$ for all $b,b' \in B$.

Let $K$ be a field containing $R$ such that the algebra $KA$ is split semisimple.
The symmetrising form $\tau$ can be extended to $KA$. If we denote by ${\rm Irr}(KA)$ the set of irreducible characters of $KA$, then there exist elements $(s_\chi)_{\chi  \in {\rm Irr}(KA)}$ in the integral closure of $R$ in $K$  such that
$$\tau=\sum_{\chi \in {\rm Irr}(KA)}\frac{1}{s_\chi} \chi$$
\cite[Theorem 7.2.6 \& Proposition 7.3.9]{GePf}. The element $s_\chi$ is the {\em Schur element} of $\chi$ with respect to $\tau$.

Schur elements are a powerful tool in the study of the representation theory of symmetric algebras. For example, we have a general semisimplicity criterion \cite[Theorem 7.4.7]{GePf}:

\begin{thm}\label{gepf}
Let $\vartheta: R \rightarrow L$ be a ring homomorphism such that $L$ is the field of fractions of $\vartheta(R)$. 
Assume that $LA:=L\otimes_R A$ is split. Then $LA$ is semisimple if and only if
$\vartheta(s_\chi) \neq 0$ for all $\chi \in {\rm Irr}(KA)$.
\end{thm}

Note that the 
Schur elements can be also  used to determine the blocks of the algebra $LA$ defined above.

\subsection{The canonical symmetrising form}
Recall the basis $\mathcal{B}$ for the Yokonuma--Hecke algebra given by (\ref{split}):
$$
\mathcal{B} =
\left\{t_1^{k_1}\ldots t_n^{k_n} g_w\,\left|\,\begin{array}{ll}w \in \mathfrak{S}_n, & k_1,\ldots,k_n \in {\Z}/d{\Z}\end{array}\right\}\right. .
$$ 
Obviously, the set 
$$
\mathcal{B}' =
\left\{g_{w'}t_1^{l_1}\ldots t_n^{l_n} \,\left|\,\begin{array}{ll}w' \in \mathfrak{S}_n, & l_1,\ldots,l_n \in {\Z}/d{\Z}\end{array}\right\}\right. 
$$ is also a basis for the Yokonuma--Hecke algebra.

\begin{prop} Define a linear map $\btau: {\rm Y}_{d,n}(q) \rightarrow \C[q,q^{-1}]$ by 
\begin{equation}\label{defn of btau}
\btau(t_1^{k_1}\ldots t_n^{k_n} g_w)=\left\{\begin{array}{ll}1, & \text{if }\,\,w=1 \,\,\text{ and }\,\, k_i
 \equiv 0 \,{\rm mod }\,d
 \,\,\text{ for all }\,\, i=1,\ldots,n ;\\
0, & \text{otherwise.}\end{array}\right.
\end{equation}
Then we have
\begin{equation}\label{mult-for}\btau(t_1^{k_1}\ldots t_n^{k_n} g_wg_{w'}t_1^{l_1}\ldots t_n^{l_n})=\left\{\begin{array}{ll}1, & \text{if }\,\, w^{-1}=w' \,\,\text{ and }\,\, k_i+l_i \equiv 0 \,{\rm mod }\,d \,\,\text{ for }\, i=1,\ldots,n ;\\
0, & \text{otherwise.}\end{array}\right.
\end{equation}
Moreover, $\btau$ is a symmetrising form and the basis dual to $\mathcal{B}$ with respect to $\btau$ is given by 
$$(t_1^{k_1}\ldots t_n^{k_n} g_w)^\vee=g_{w^{-1}}t_1^{d-k_1}\ldots t_n^{d-k_n}.$$
\end{prop}

\begin{proof} 
For the proof of the multiplication formula (\ref{mult-for}), we proceed by induction on the length $\ell(w)$ of $w$.
If $w=1$, there is nothing to prove. Now let $\ell(w)>0$ and choose a transposition $s_i \in \mathfrak{S}_n$ such that $\ell(ws_i)< \ell(w)$. By (\ref{rightmulti}), 
we have
$g_w=g_{(ws_i)s_i}=g_{ws_i}g_{s_i}$, and hence, $g_wg_{w'}=g_{ws_i}g_{s_i}g_{w'}$. Now we distinguish two cases:

\emph{Case 1. $\ell(s_iw') > \ell(w')$: }  By (\ref{leftmulti}), we have $g_{s_i}g_{w'}=g_{s_iw'}$. 
Note that $\ell(ws_i)< \ell(w)$ implies that $\ell(s_iw^{-1})< \ell(w^{-1})$. Thus 
we must have $w^{-1} \neq w'$ and, hence, $(ws_i)^{-1} \neq s_iw'$. By induction, we have 
$$\btau(t_1^{k_1}\ldots t_n^{k_n} g_wg_{w'}t_1^{l_1}\ldots t_n^{l_n})=\btau(t_1^{k_1}\ldots t_n^{k_n} g_{ws_i}g_{s_iw'}t_1^{l_1}\ldots t_n^{l_n})=0.$$

\emph{Case 2. $\ell(s_iw') < \ell(w')$: }  By (\ref{leftmulti}), we have $g_{s_i}g_{w'}=g_{s_iw'}+(q-q^{-1})e_ig_{w'}$, and hence,
$$\btau(t_1^{k_1}\!\ldots t_n^{k_n} g_wg_{w'}t_1^{l_1}\!\ldots t_n^{l_n})=\btau(t_1^{k_1}\!\ldots t_n^{k_n} g_{ws_i}g_{s_iw'}t_1^{l_1}\!\ldots t_n^{l_n})
+(q-q^{-1})\btau(t_1^{k_1}\!\ldots t_n^{k_n} g_{ws_i}e_ig_{w'}t_1^{l_1}\!\ldots t_n^{l_n}).$$
Due to the length inequalities, we have $(ws_i)^{-1} \neq w'$, and induction hypothesis yields that
$$\btau(t_1^{k_1}\ldots t_n^{k_n} g_{ws_i}e_ig_{w'}t_1^{l_1}\ldots t_n^{l_n})=0.$$
Now, if $w^{-1} \neq w'$, then $(ws_i)^{-1} \neq s_iw'$, and by induction,
$$\btau(t_1^{k_1}\ldots t_n^{k_n} g_{ws_i}g_{s_iw'}t_1^{l_1}\ldots t_n^{l_n})=0.$$
On the other hand, if $w^{-1} = w'$, then $(ws_i)^{-1} = s_iw'$, and by induction,
$$\btau(t_1^{k_1}\ldots t_n^{k_n} g_{ws_i}g_{s_iw'}t_1^{l_1}\ldots t_n^{l_n})=\left\{\begin{array}{ll}1, & \text{if }\,\, k_i+l_i \equiv 0 \,{\rm mod }\,d \,\,\text{ for all }\,\, i=1,\ldots,n ;\\
0, & \text{otherwise.}\end{array}\right.$$

Now, $\btau$ is a trace function, because
$$\btau(t_1^{k_1}\ldots t_n^{k_n} g_wg_{w'}t_1^{l_1}\ldots t_n^{l_n})=\btau (g_{w'}t_1^{l_1}\ldots t_n^{l_n}t_1^{k_1}\ldots t_n^{k_n} g_w),$$
which is obtained using Formula (\ref{mult-for}) together with
$$\btau (g_{w'}t_1^{l_1}\ldots t_n^{l_n}t_1^{k_1}\ldots t_n^{k_n} g_w)=
\btau (g_{w'}t_1^{k_1+l_1}\ldots t_n^{k_n+l_n} g_w)=
\btau(t_{w'(1)}^{k_1+l_1}\ldots t_{w'(n)}^{k_n+l_n} g_{w'}g_w).$$

Finally, let $\mathcal{B}^\vee$ be the set 
$$
\mathcal{B}^\vee =
\left\{g_{w^{-1}}t_1^{d-k_1}\ldots t_n^{d-k_n} \,\left|\,\begin{array}{ll}w \in \mathfrak{S}_n, & k_1,\ldots,k_n \in {\Z}/d{\Z}\end{array}\right\}\right. .
$$ Then  $\mathcal{B}^\vee$  is also a basis for ${\rm Y}_{d,n}(q)$, and we have
$$\btau(t_1^{k_1}\ldots t_n^{k_n} g_wg_{w^{'-1}}t_1^{d-l_1}\ldots t_n^{d-l_n})=
\left\{\begin{array}{ll}1, & \text{if }\,\, w=w' \,\,\text{ and }\,\, k_i
\equiv l_i \,{\rm mod }\,d
 \,\,\text{ for all }\,\, i=1,\ldots,n ;\\
0, & \text{otherwise.}\end{array}\right..$$
This means that the bilinear form ${\rm Y}_{d,n}(q)  \times {\rm Y}_{d,n}(q) \rightarrow  \C[q,q^{-1}],\, (a,b) \mapsto \btau(ab)$ is non-degenerate, and 
that $\mathcal{B}^\vee$ is the dual basis to $\mathcal{B}$ with respect to $\btau$.
\end{proof}

We will call $\btau$ the {\em canonical symmetrising form on} ${\rm Y}_{d,n}(q)$, because  $\btau$ becomes the canonical symmetrising form on the group algebra of $G(d,1,n)$ for $q=\pm 1$.  

\begin{rem}{\rm
The map $\btau$ is a Markov trace on ${\rm Y}_{d,n}(q)$ in the sense of Juyumaya \cite{ju3}, with all parameters equal to zero.} 
\end{rem}

Finally, let $\alpha \in {\rm Y}_{d,n}(q)$. The element $\alpha$ can be written in a unique way as a linear combination of elements of $\mathcal{B}$, that is, $\alpha= \sum_{b \in \mathcal{B}} a_b b\,$
for some unique 
$a_b \in \C[q,q^{-1}].$ Then
$\btau(\alpha)=a_1$, the coefficient in front of the unit element of ${\rm Y}_{d,n}(q)$.

\subsection{Schur elements for ${\rm Y}_{d,n}(q)$} 
Recall that $\mathcal{P}(d,n)$ denotes the set of all $d$-partitions of $n$. 
The canonical symmetrising form $\btau$ can be extended to the split semisimple algebra
$\C(q){\rm Y}_{d,n}(q)$. Then we have
$$\btau=\sum_{\blambda \in \mathcal{P}(d,n)} \frac{1}{s_{\blambda}}\chi_{\blambda},$$
where $\chi_{\blambda}$ is the character of the irreducible representation $V_{\blambda}$ defined in Proposition \ref{prop-rep}, and 
$s_{\blambda} \in \C[q,q^{-1}]$ is the  Schur element of $V_{\blambda}$ with respect to $\btau$. 

Let $\blambda \in \mathcal{P}(d,n)$, and let $\cT$ be a standard $d$-tableau of shape $\blambda$.
 Let $E_{_{\cT}}$ be the minimal idempotent of $\C(q){\rm Y}_{d,n}(q)$ corresponding to $\cT$ (see Subsection \ref{ET}). Then we have 
\begin{equation}\label{schur}
\btau(E_{_{\cT}}) = \frac{1}{s_{\blambda}}.
\end{equation}
We will use the above formula in order to calculate the Schur elements for ${\rm Y}_{d,n}(q)$, together with Formula (\ref{idem-JM}) for $E_{_{\cT}}$, which we will slightly modify here.

 Let $\btheta$ be the $d$-node of $\cT$ with the number $n$ in it. Let ${\mathcal{U}}$ be the standard $d$-tableau obtained from $\cT$ by removing the $d$-node $\btheta$ and let $\bmu$ be the shape of ${\mathcal{U}}$. We recall the inductive formula 
 (\ref{idem-JM}) for the minimal idempotent $E_{_{\cT}}$ of $\C(q){\rm Y}_{d,n}(q)$ corresponding to $\cT$:
 $$E_{_{\cT}}=E_{_{{\mathcal{U}}}}\prod_{ \begin{array}{l}\scriptstyle{\btheta'\in 
{\mathcal{E}}_+(\bmu)}\\\scriptstyle{\cc(\btheta')\neq \cc(\btheta)}\end{array}}\hspace{-0.2cm} \frac{J_n-\cc(\btheta')}{\cc(\btheta)-\cc(\btheta')}\prod_{ \begin{array}{l}\scriptstyle{\btheta'\in 
{\mathcal{E}}_+(\bmu)}\\\scriptstyle{\pos(\btheta')\neq \pos(\btheta)}\end{array}  }\hspace{-0.2cm}\frac{t_n-\xi_{\pos(\btheta')}}{\xi_{\pos(\btheta)}-\xi_{\pos(\btheta')}}\ 
$$
with $E_{_{\cT_0}}=1$ for the unique $d$-tableau $\cT_0$ of size $0$. Now, let $\gad=\{\xi_1,\dots,\xi_d\}$ be the set of all $d$-th roots of unity (ordered arbitrarily). We have
$$E_{_{{\mathcal{U}}}} \prod_{ \begin{array}{l}\scriptstyle{\btheta'\in 
{\mathcal{E}}_+(\bmu)}\\\scriptstyle{\pos(\btheta')\neq \pos(\btheta)}\end{array}  }\hspace{-0.2cm}\frac{t_n-\xi_{\pos(\btheta')}}{\xi_{\pos(\btheta)}-\xi_{\pos(\btheta')}}
 =  \sum_{ \begin{array}{l}\scriptstyle{\btheta'\in 
{\mathcal{E}}_+(\bmu)}\\\scriptstyle{\pos(\btheta')=\pos(\btheta)}\end{array}  }
E_{_{\mathcal{U}\cup\{\btheta'\}}}\,\,\,=\,\,E_{_{{\mathcal{U}}}} \prod_{\begin{array}{l}\scriptstyle{\xi\in 
\gad}\\\scriptstyle{\xi\neq \xi_{\pos(\btheta)}}\end{array}  }\hspace{-0.2cm}\frac{t_n-\xi}{\xi_{\pos(\btheta)}-\xi}$$
We deduce that (\ref{idem-JM}) is equivalent to
\begin{equation}\label{idem-JM2}
E_{_{\cT}}=E_{_{{\mathcal{U}}}}\prod_{ \begin{array}{l}\scriptstyle{\btheta'\in 
{\mathcal{E}}_+(\bmu)}\\\scriptstyle{\cc(\btheta')\neq \cc(\btheta)}\\\scriptstyle{\pos(\btheta')=\pos(\btheta)}\end{array}}\hspace{-0.2cm} \frac{J_n-\cc(\btheta')}{\cc(\btheta)-\cc(\btheta')}\prod_{\begin{array}{l}\scriptstyle{\xi\in 
\gad}\\\scriptstyle{\xi\neq \xi_{\pos(\btheta)}}\end{array}  }\hspace{-0.2cm}\frac{t_n-\xi}{\xi_{\pos(\btheta)}-\xi}\ .
\end{equation}

In Equation (\ref{idem-JM2}) set
$$E_n^{\cc}:=\prod_{\ \begin{array}{l}\scriptstyle{\btheta'\in 
{\mathcal{E}}_+(\bmu)}\\\scriptstyle{\cc(\btheta')\neq \cc(\btheta)}\\\scriptstyle{\pos(\btheta')=\pos(\btheta)}\end{array}}\hspace{-0.2cm} \frac{J_n-\cc(\btheta')}{\cc(\btheta)-\cc(\btheta')}
\,\,\,\,\,\text{ and }\,\,\,\,\,
E_n^{\pos}:=\prod_{ \begin{array}{l}\scriptstyle{\xi\in 
\gad}\\\scriptstyle{\xi\neq \xi_{\pos(\btheta)}}\end{array}  }\hspace{-0.2cm}\frac{t_n-\xi}{\xi_{\pos(\btheta)}-\xi}\ .
$$
The elements $E_n^{\cc}$ and $E_n^{\pos}$ are both idempotents, and they commute with each other (due to the commutativity of the Jucys--Murphy elements).
We repeat this process for the idempotent $E_{_{{\mathcal{U}}}}$ and so on, that is, we set, for all $i=1,\dots,n-1$, 
$$E_i^{\cc}:=\prod_{ \begin{array}{l}\scriptstyle{\btheta'\in 
{\mathcal{E}}_+(\bmu_{i-1})}\\\scriptstyle{\cc(\btheta')\neq \cc(\btheta_i)}\\\scriptstyle{\pos(\btheta')=\pos(\btheta_i)}\end{array}}\hspace{-0.2cm} \frac{J_i-\cc(\btheta')}{\cc(\btheta_i)-\cc(\btheta')}
\,\,\,\,\,\text{ and }\,\,\,\,\,
E_i^{\pos}:=\prod_{ \begin{array}{l}\scriptstyle{\xi\in 
\gad}\\\scriptstyle{\xi\neq \xi_{\pos(\btheta_i)}}\end{array}  }\hspace{-0.2cm}\frac{t_i-\xi}{\xi_{\pos(\btheta_i)}-\xi}\ ,
$$
where $\btheta_i$ is the $d$-node of $\cT$ with the number $i$ in it, and $\bmu_{i-1}$ is the shape of the standard $d$-tableau obtained from $\cT$ by removing the $d$-nodes $\btheta_i,\btheta_{i+1},\ldots,\btheta_n$.

By construction, the inductive formula (\ref{idem-JM2}) reads now
$$E_{_{\cT}}=E_1^{\cc}E_1^{\pos}\,E_2^{\cc}E_2^{\pos}\ldots E_{n-1}^{\cc}E_{n-1}^{\pos}\,E_{n}^{\cc}E_{n}^{\pos}. $$
Set
$$E_{_{\cT}}^{\cc}:=E_{1}^{\cc}E_{2}^{\cc}\ldots E_{n-1}^{\cc}E_n^{\cc} \,\,\,\,\,\text{and}\,\,\,\,\,
E_{_{\cT}}^{\pos}:=E_{1}^{\pos}E_{2}^{\pos}\ldots E_{n-1}^{\pos}E_n^{\pos}.$$
Then
\begin{equation}\label{form-Et} E_{_{\cT}}=E_{_{\cT}}^{\pos}E_{_{\cT}}^{\cc}\,.\end{equation}
The idempotent $E_{_{\cT}}^{\pos}$ determines the position of each $d$-node, while $E_{_{\cT}}^{\cc}$ 
determines 
the content of each $d$-node (and thus, its position in the Young diagram of the partition specified by $E_{_{\cT}}^{\pos}$) in $\cT$. By definition of $E_{_{\cT}}^{\pos}$, for all $i=1,\ldots,n$, we have that 
\begin{equation}\label{comm-Epos}
t_iE_{_{\cT}}^{\pos}=E_{_{\cT}}^{\pos}t_i=\xi_{\pos(\cT|i)}E_{_{\cT}}^{\pos},
\end{equation}
 and hence,
 \begin{equation}\label{multi-Epos}
  e_{i,k}E_{_{\cT}}^{\pos}=E_{_{\cT}}^{\pos}e_{i,k}=\left\{ \begin{array}{ll} E_{_{\cT}}^{\pos}, & \text{ if }\,\, \pos(\cT|i)=\pos(\cT|k)\\
 0, & \text{ if }\,\, \pos(\cT|i)\neq\pos(\cT|k)\\ \end{array}\right..
 \end{equation}
 Finally, it is easy to check that 
 \begin{equation}\label{d^n}
 \btau(E_{_{\cT}}^{\pos})=   \prod_{\btheta\in\blambda} \prod_{ \begin{array}{l}\scriptstyle{\xi\in 
\gad}\\\scriptstyle{\xi\neq \xi_{\pos(\btheta)}}\end{array}  }\hspace{-0.2cm}\frac{-\xi}{\xi_{\pos(\btheta)}-\xi}\ =
   \prod_{\btheta\in\blambda} \frac{1}{d} 
  =\frac{1}{d^n}, 
 \end{equation}
since $\prod_{\xi \in \gad\backslash\{1\}}(1-\xi)=d$.

\vskip .1cm
Before we determine the Schur elements for ${\rm Y}_{d,n}(q)$ with respect to $\btau$, we introduce the following notation:
Let $\lambda=(\lambda_1,\dots,\lambda_k)$ be a partition and let $\lambda'=(\lambda'_1,\dots,\lambda'_l)$ be the conjugate partition of $\lambda$. 
We set 
$$\eta(\lambda):=\sum_{i=1}^k (i-1)\lambda_i= \frac{1}{2}\sum_{j=1}^l{\lambda'_j(\lambda'_j-1)}\ .$$ 
Now let $\blambda=(\blambda^{(1)},\ldots,\blambda^{(d)})$ be a $d$-partition of $n$. We set $\eta(\blambda):=\sum_{i=1}^d \eta(\blambda^{(i)})$. Moreover, recall that, for any $d$-node $\btheta$ of $\blambda$, we denote by ${\rm hl}(\btheta)$ the hook length of $\btheta$ (see (\ref{hl})). 

\begin{prop}\label{Schur elements}
Let $\blambda \in \mathcal{P}(d,n)$.
We have
 \begin{equation}\label{schur-multipartition}
s_{\blambda}= d^n\,q^{-2\eta(\blambda)}\, \prod_{\btheta\in\blambda} [{\rm hl}(\btheta)]_{q^2}\ ,
\end{equation}
where, for all $h \in \mathbb{N}$, $[h]_{q^2} := (q^{2h}-1)/(q^2-1) = q^{2(h-1)} + q^{2(h-2)} +\cdots+ q^2 + 1$.
\end{prop}

\begin{proof}
Let $n_i$ be the size of the partition $\blambda^{(i)}$, for all $i=1,\ldots,d$. Let $\cT$ be a standard $d$-tableau of shape $\blambda$.
In order to facilitate the computation of $\btau(E_{_{\cT}})$, we will assume that 
\begin{center}
$\pos_1=\cdots=\pos_{n_1}=1,\ $ $\pos_{n_1+1}=\cdots=\pos_{n_1+n_2}=2,\ $ 
$\ldots\,,\ $ $\pos_{n_1+\cdots+n_{d-1}+1}=\cdots=\pos_{n}=d$,
\end{center}
where we have set, for brevity, $\pos_j:=\pos(\cT|j)$ for all $j=1,\dots,n$.

Let $i \in \{1,\ldots,d\}$. Set $m_i:=n_1 + \cdots + n_{i-1}$.  Let $\mathcal{A}^{(i)}$ be the subalgebra of $\C(q){\rm Y}_{d,n}(q)$ generated by the elements
$$t_{m_i+1}, t_{m_i+2},\ldots,t_{m_i+n_i}, g_{m_i+1},g_{m_i+2},\ldots,g_{m_i+n_i-1}.$$
Then $\mathcal{A}^{(i)}$ is isomorphic to $\C(q){\rm Y}_{d,n_i}(q)$.
By the commutativity of  $t_1,\ldots,t_n$,
the element $E_{_{\cT}}^{\pos}$ commutes
with the elements $t_{m_i+1}, t_{m_i+2},\ldots,t_{m_i+n_i}$. Now, 
since
$$\pos_{m_i+k}=\pos_{m_i+k+1}=i$$
for all $k=1,2,\ldots,n_i-1$, we have $E_{_{\cT}}^{\pos}g_{m_i+k}=g_{m_i+k}E_{_{\cT}}^{\pos}$, due to the defining relation (\ref{modular})$\mathrm{(f}_2)$.
Thus,
the element $E_{_{\cT}}^{\pos}$ commutes with all elements of $\mathcal{A}^{(i)}$. Moreover, due to (\ref{multi-Epos}), we have
\begin{equation}\label{hecke-quadr}E_{_{\cT}}^{\pos} g_{m_i+k}^2=  E_{_{\cT}}^{\pos} \left(1+(q-q^{-1}) g_{m_i+k} \right).\end{equation}
As a consequence, we obtain a $\C(q)$-algebra epimorphism
$\varphi_i$ from the algebra $E_{_{\cT}}^{\pos} \mathcal{A}^{(i)} $ to the Iwahori--Hecke algebra $\C(q)\mathcal{H}_{n_i}(q)$ of type $A$,
whose generators we denote by $G_{m_i+1},\ldots,G_{m_i+n_i-1}$, defined by
\begin{equation}\label{phi}
\varphi_i(E_{_{\cT}}^{\pos})=1,\,\,\,\,\varphi_i(E_{_{\cT}}^{\pos}t_{m_i+j})=\xi_i\,\,\, \text{ and }\,\,\,\varphi_i(E_{_{\cT}}^{\pos}g_{m_i+k})=G_{m_i+k},
\end{equation}
for all $j=1,2,\ldots,n_i$ and for all $k=1,2,\ldots,n_i-1$.

Due to (\ref{form-JM}) and (\ref{multi-Epos}), we also have that, for all $j=1,2,\ldots,n_i$,  the element $E_{_{\cT}}^{\pos} J_{m_i+j}$ can be written as the product of $E_{_{\cT}}^{\pos}$ with a linear combination of elements of the form $g_w$, where $w \in \langle s_{m_i+1}, s_{m_i+2},\ldots,s_{m_i+n_i-1} \rangle \cong \mathfrak{S}_{n_i}$. As a consequence, the same holds for the element
$$E^{(i)}:=E_{_{\cT}}^{\pos} \prod_{j=1}^{n_i}E^{\cc}_{m_i+j}\,.$$
That is, we have a linear combination $\tilde{E}^{(i)}$ of elements of the form $g_w$, where $w \in \langle s_{m_i+1}, s_{m_i+2},\ldots,$ $s_{m_i+n_i-1} \rangle$, such that
\begin{equation}\label{Ei}
E^{(i)}=E_{_{\cT}}^{\pos}\tilde{E}^{(i)}.
\end{equation}
In particular, $E^{(i)}\in E_{_{\cT}}^{\pos} \mathcal{A}^{(i)}$. 
Moreover, we have
\begin{equation}\label{Eti}\varphi_i(E^{(i)})= e_{T_i}\ ,\end{equation}
where $T_i$ is the standard tableau of shape $\blambda^{(i)}$ such that $\cc(T_i|j)=\cc(\cT|m_i+j)$ for all $j=1,\ldots,n_i$,
and $e_{T_i}$ is the corresponding minimal idempotent in $\C(q)\mathcal{H}_{n_i}(q)$ (for an inductive description of $e_{T_i}$ in terms of the Jucys--Murphy elements of $\mathcal{H}_{n_i}(q)$, see, for example, \cite{IMO} or \cite[\S 5]{OPdA3}).

Now, for any $w \in \langle s_{m_i+1}, s_{m_i+2},\ldots,s_{m_i+n_i-1} \rangle$, we can define an element $G_w$ in  $\mathcal{H}_{n_i}(q)$
 in the same way that we defined the element $g_w$ in ${\rm Y}_{d,n}(q)$. The elements $G_w$, $w\in\langle s_{m_i+1}, s_{m_i+2},\ldots,s_{m_i+n_i-1} \rangle$, form the ``standard'' basis of $\mathcal{H}_{n_i}(q)$ as a $\C[q,q^{-1}]$-module, and we have
 $G_w=\varphi_i(E_{_{\cT}}^{\pos}g_w)$. 
Further, the algebra $\mathcal{H}_{n_i}(q)$  is endowed with a canonical symmetrising form $\tau_{i}$, such that, for all $a \in  \mathcal{H}_{n_i}(q)$, $\tau_{i}(a)$ is equal to the coefficient of $1$ when $a$ is expressed as a linear combination of the standard basis elements $G_w$.
From (\ref{Ei}) and (\ref{Eti}), we deduce that 
\begin{equation}\label{schur1}\btau(\tilde{E}^{(i)})=\tau_{i}(e_{T_i})\ .
\end{equation}

Note now that the expression (\ref{form-Et}) for $E_{_{\cT}}$ can be rewritten as
\begin{equation}\label{schur2}E_{_{\cT}}=E_{_{\cT}}^{\pos}{E}^{(1)}{E}^{(2)}\ldots{E}^{(d)}=E_{_{\cT}}^{\pos}\tilde{E}^{(1)}\tilde{E}^{(2)}\ldots\tilde{E}^{(d)}\ .\end{equation}
As we have seen, for all $i=1,\ldots,d$, $\tilde{E}^{(i)}$ is a linear combination of elements of the form $g_w$, where $w \in \langle s_{m_i+1}, s_{m_i+2},\ldots,s_{m_i+n_i-1} \rangle$. Note that for $i,i' \in \{1,\dots,d\}$ such that $i < i'$, the elements of $\langle s_{m_i+1}, s_{m_i+2},\ldots,s_{m_i+n_i-1} \rangle$ commute with the elements of $\langle s_{m_{i'}+1}, s_{m_{i'}+2},\ldots,s_{m_{i'}+n_{i'}-1} \rangle$, since 
$m_{i'}\geq m_i+n_i$. 
This together with 
(\ref{schur1}) and (\ref{schur2})  imply that 
$$\btau(E_{_{\cT}})= \btau(E_{_{\cT}}^{\pos}) \btau(\tilde{E}^{(1)})\btau(\tilde{E}^{(2)})\ldots \btau(\tilde{E}^{(d)})
=\btau(E_{_{\cT}}^{\pos})\tau_1(e_{T_1})\tau_2(e_{T_2})\dots \tau_d(e_{T_d})\ .$$

For $i=1,\dots,d$, the element $\tau_i(e_{T_i})$ is equal to $1/s_{\blambda^{(i)}}$, where $s_{\blambda^{(i)}}$ is the Schur element
for $\mathcal{H}_{n_i}(q)$ associated to the partition $\blambda^{(i)}$ with respect to $\tau_i$. The
form of the Schur element $s_{\blambda^{(i)}}$ is already known (see, for example, \cite[Theorem 3.2]{ChJa}):
 \begin{equation}\label{schur-partition}
 s_{\blambda^{(i)}} = q^{-2\eta(\blambda^{(i)})}\, \prod_{\theta \in \blambda^{(i)}} [{\rm hl}(\theta)]_{q^2}.
 \end{equation}

Using (\ref{d^n}), we conclude that
$$\btau(E_{_{\cT}})= \frac{1}{d^n s_{\blambda^{(1)}} s_{\blambda^{(2)}} \cdots s_{\blambda^{(d)}}},$$
 whence we deduce that
$$s_{\blambda}=d^n s_{\blambda^{(1)}} s_{\blambda^{(2)}} \cdots s_{\blambda^{(d)}},$$
where we take $s_\emptyset:=1$.
Applying (\ref{schur-partition}) yields  (\ref{schur-multipartition}). 
\end{proof}

\begin{rem}{\rm
Theorem \ref{gepf}, combined with the description of the Schur elements for ${\rm Y}_{d,n}(q)$ 
given by the proposition above, yields the semisimplicity criterion that we proved in the previous section.}
 \end{rem}
 
 Finally, let $\btheta=(x,y,k)$ be a $d$-node of the $d$-partition $\blambda$ of $n$. We define the  \emph{classical content} of $\btheta$ to be the integer $\ccc(\btheta):=y-x$. The following corollary of Proposition \ref{Schur elements} gives an alternative formula for the Schur elements of  ${\rm Y}_{d,n}(q)$, involving the classical contents of the $d$-nodes instead of the function $\eta$.

\begin{cor}  
Let $\blambda \in \mathcal{P}(d,n)$. We have
\begin{equation}\label{schur-multipartition2}
s_{\blambda}= d^n\, \prod_{\btheta\in\blambda} q^{\ccc(\btheta)}\{{\rm hl}(\btheta)\}_q\ ,
\end{equation}
where, for all $h \in \mathbb{N}$, $\{h\}_q := (q^{h}-q^{-h})/(q-q^{-1}) = q^{(h-1)} + q^{(h-2)} +\cdots+ q^{-(h-2)} + q^{-(h-1)}$. 
\end{cor}
\begin{proof}
In order to prove that (\ref{schur-multipartition2}) is equivalent to (\ref{schur-multipartition}), it is enough to show that, for any {partition $\lambda$}, we have
\begin{equation}\label{form-part}
2\eta(\lambda)=\sum_{\theta\in\lambda}\bigl({\rm hl}(\theta)-1-\ccc(\theta)\bigr)\ .
\end{equation}
Let $\lambda=(\lambda_1,\dots,\lambda_k)$ be a partition of $n$ and let $\lambda'=(\lambda'_1,\dots,\lambda'_l)$ be the conjugate partition. Then 
$$2\eta(\lambda)=\sum_{j=1}^l\lambda'_j(\lambda'_j-1)=\sum_{j=1}^l(\lambda'_j)^2\,-\,n\ .
$$
If $\theta=(i,j)$ is a node of $\lambda$, then ${\rm hl}(\theta)=\lambda_i-i+\lambda_j'-j+1$ and $\ccc(\theta)=j-i$. We obtain that
$$\sum_{\theta\in\lambda}\bigl({\rm hl}(\theta)-1-\ccc(\theta)\bigr)=\sum_{i=1}^k\sum_{j=1}^{\lambda_i}(\lambda'_j+\lambda_i-2j)=\sum_{i=1}^k\sum_{j=1}^{\lambda_i}\lambda'_j+\sum_{i=1}^k\bigl(\lambda_i^2-\lambda_i(\lambda_i+1)\bigr)
=\sum_{i=1}^k\sum_{j=1}^{\lambda_i}\lambda'_j-n\ .
$$
Now, since $\lambda'_{j}=\#\{m\,|\,1 \leq m \leq k \text{ such that } \lambda_m\geq j\}$, it is easy to see that $$\sum\limits_{i=1}^k\sum\limits_{j=1}^{\lambda_i}\lambda'_j=\sum\limits_{j=1}^{l}(\lambda'_j)^2.$$ Therefore, (\ref{form-part}) holds.
\end{proof}

\end{document}